\documentclass[11pt]{article}
\usepackage[utf8]{inputenc}
\usepackage{soul}
\usepackage{comment}
\usepackage{pgfplots}
\usepackage{tikz}
\usepackage{authblk}

\newcommand{\required}[1]{\section*{\hfil \sharsokol2013advanced\hfil}}

\usepackage{verbatim,url}
\usepackage{graphicx, epsfig,subfloat,subfig}
\graphicspath{{LRA_figs_1/}}
\usepackage{bbm}
\usepackage{amssymb,amsmath,amsfonts,amsthm}
\usepackage{algorithm}
\usepackage{algpseudocode}
\usepackage{upref}

\newcommand{\Beq}{\begin{equation}}
\newcommand{\Eeq}{\end{equation}}
\newcommand{\be}{\begin{equation}}
\newcommand{\ee}{\end{equation}}
\newcommand{\bs}{\begin{split}}
\newcommand{\beq}{\begin{equation*}}
\newcommand{\eeq}{\end{equation*}}
\newcommand{\bal}{\begin{align}}
\newcommand{\eal}{\end{align}}

\newtheorem{theorem}{Theorem}
\numberwithin{theorem}{section}
\newtheorem{proposition}[theorem]{Proposition}

\newcommand{\Hb}{\mathbb{H}}

\numberwithin{equation}{section}

\newcommand{\bel}[1]{\begin{equation}\label{#1}}

\newcommand{\eel}[1]{{\label{#1}\end{equation}}}

\newcommand{\Hr}{\mathbb{H}_r}
\newcommand{\Lambar}{\bar{\Lambda}}
\newcommand{\Gg}{\tilde{\Lambda}_\gamma}
\newcommand{\Xig}{\Xi_\gamma}
\newcommand{\norm}[1]{\|#1\|}
\newcommand{\ip}[2]{\langle #1,\, #2 \rangle}

\renewcommand{\bar}[1]{\overline{#1}}

\newtheorem{lemma}[theorem]{Lemma}

\newtheorem{corollary}[theorem]{Corollary}

\theoremstyle{definition}

\newtheorem{remark}{Remark}

\usepackage{hyperref}
\usepackage{tikz-cd}
\usepackage[all]{xy}
\usepackage{graphicx}
\hypersetup{
    colorlinks=true,
    linkcolor=blue,
    filecolor=magenta,      
    urlcolor=cyan,
}

\usepackage{xcolor}

\usepackage{authblk} 

\title{An Efficient Bayesian Framework for Uncertainty Quantification in Nonlinear Imaging Inverse Problems}
\author[$a$]{Anuj Abhishek} 
\author[$b$]{Sakshi Arya}
\author[$c$] {Madhu Gupta}
\affil[$a$]{\small Department of Mathematics, Applied Mathematics and Statistics, Case Western Reserve University, USA\par anuj.abhishek@case.edu}
\affil[$b$]{\small Department of Mathematics, Applied Mathematics and Statistics, Case Western Reserve University, USA\par sakshi.arya@case.edu}
\affil[$c$]{\small Department of Mathematics, IIT Gandhinagar, India\par madhu.gupta@iitgn.ac.in}

\date{}
\begin{document}
\maketitle
\begin{abstract} 
Bayesian methods provide a natural framework for estimating a parameter in non-linear inverse problems and quantifying uncertainty in the estimation. However, when the forward model for such non-linear inverse problems is given by some Partial Differential Equation (PDE), Bayesian inference is typically carried out by resorting to MCMC methods. Since each MCMC iteration requires solving a PDE, these methods become computationally expensive and are often impractical for large-scale imaging problems. In this work, we develop a computationally efficient Bayesian framework for two such nonlinear imaging inverse problems: Quantitative Photoacoustic Tomography (QPAT) and Electrical Impedance Tomography (EIT). Building on a recently proposed two-stage pushforward methodology, we first formulate a Bayesian regression problem for an auxiliary variable whose posterior is available in closed form. This posterior is then pushed forward through a deterministic reconstruction map to obtain a posterior on the unknown parameter, avoiding MCMC sampling. We give a rigorous measure-theoretic justification to interpret the induced posterior as a Bayesian posterior and derive posterior contraction rates for both QPAT and EIT. Numerical results show that the proposed method provides accurate reconstructions and reliable uncertainty estimates at a arguably lower computational cost than standard Bayesian approaches.
\end{abstract}

\section{Introduction}\label{intro}

Non-linear inverse problems of coefficient identification in Partial Differential Equations (PDEs) arise naturally in  a wide range of imaging paradigms such as in medical imaging, geophysical imaging, and non-destructive testing, see e.g., \cite{kaipio2006statistical,kuchment_book,tarantola2005inverse,U_sur}.
In this article, we study two such prototypical non-linear inverse problems, namely, Quantitative Photo-Acoustic Tomography (QPAT), \cite{afkham2024bayesian,tarva_12} and Electrical Impedance Tomography (EIT),\cite{borcea2002electrical,cheney1999electrical,somersalo_92}.  Under the diffusion-approximation based formulation of QPAT , the objective is to recover a spatially varying optical absorption parameter of a medium based on the measurement of absorbed optical energy density within the medium. On the other hand, in the inverse problem of EIT, the objective is to recover the internal conductivity of the medium by injecting currents and measuring the corresponding voltages via electrodes placed only on the boundary of the medium. These two problems represent two fundamentally different types of non-linear inverse problems for recovering spatially varying parameters of a medium from (i) measurements accessible within the medium, e.g., in QPAT, and (ii) measurements that can be made only on the boundary of the medium, e.g., in EIT. Such inverse problems are typically highly ill-posed; i.e., the solutions are highly susceptible to noise in the data. To deal with this high-degree of ill-posedness in the problem, all practically useful reconstruction methods employ some form of regularization \cite{engl1996regularization,jin2012reconstruction,jin_12}.

The Bayesian formulation of such inverse problems can also be seen under such a regularization lens; where the choice of the prior measure acts as a regularization tool \cite{calvetti_somersalo_review}. It is thus not surprising that Bayesian formulation of inverse problems are typically well-posed under standard regularity assumptions, \cite{DashtiStuart2013,Stuart2010}.
To present the main ideas in an abstract Bayesian framework briefly, let $\mathcal{X}$ and $\mathcal{Y}$ be separable Hilbert spaces and  $G:\mathcal{X}\to\mathcal{Y}$ be a measurable map, and consider the inverse problem of recovering an unknown parameter $v\in\mathcal{X}$ from noisy data
\begin{align}\label{eq:1.3}
y = G(v)+\eta,
\end{align}
where $y\in\mathcal{Y}$ denotes the observation and $\eta\in\mathcal{Y}$ is observational noise. In the Bayesian formulation, the unknown parameter is modeled as a random variable with prior law $\mu_0$, and the solution of the inverse problem is the posterior distribution (or, more generally a measure) of $v$ given the data $y$. In finite dimensions, the posterior can often be written using Bayes' rule in terms of densities which are the Radon-Nikodym derivatives of the prior and posterior measures respectively with respect to the Lebesgue measure on that finite dimensional space. In infinite-dimensional settings, however, there is no translation-invariant Lebesgue measure on function spaces, so Bayes' theorem must be formulated directly at the level of measures. Under standard assumptions (see  \cite[Theorem 14.]{DashtiStuart2013}), the posterior measure $\mu^y$ can be shown to be absolutely continuous with respect to the prior $\mu_0$ and thus admits the Radon--Nikodym representation
\begin{align}\label{eq:intro-bayes}
\frac{d\mu^y}{d\mu_0}(v) =\frac{1}{Z_y}\exp\bigl(-\Phi(v,y)\bigr),
\end{align}
where $\Phi$ is the negative log-likelihood and $Z_y$ is a normalizing constant. Thus, within the Bayesian viewpoint, the posterior measure provides the natural notion of solution to the inverse problem. Under additional regularity assumptions on the likelihood functional, well posedness of the solution can be shown. In particular, one can show that small perturbations in the observed data result in small perturbations in the posterior measure. For details, we refer the reader to \cite{DashtiStuart2013,Stuart2010}.

The Bayesian approach to inverse problems has become increasingly popular in recent decades and has been applied successfully to the study of many inverse problems, such as \cite{cal_14,kaipio2006statistical, Stuart2010}. 
 Unlike classical deterministic reconstruction methods, which yield a single 
point estimate of the unknown parameter, the Bayesian approach encodes the solution as a full posterior probability distribution over the parameter space. This posterior distribution provides not only a point estimate (e.g., 
the posterior mean or MAP estimate), but also credible intervals and 
posterior predictive distributions that quantify reconstruction confidence 
and identify regions where the data are less informative. This is of much practical importance in imaging applications such as QPAT and EIT, where reliable uncertainty estimates are critical for assessing the trustworthiness of reconstructed material 
parameters.

Due to the inherent non-linearity of such problems, even when working with tractable Gaussian process priors, the posterior distribution (or, more generally measure) does not admit any closed form expression. Thus in order to estimate the parameter and do meaningful uncertainty quantification for the estimated parameter, one resorts to MCMC-sampling based methods such as pCN, \cite{cotter2013mcmc}. 
While such MCMC-based approaches are in principle very powerful, they 
suffer from a key weakness when the forward problem is modelled by a 
Partial Differential Equation (PDE): each likelihood evaluation entails 
a call to a PDE solver. Due to the highly non-linear nature of the 
inverse problems considered here, MCMC chains tend to mix poorly, 
requiring millions of iterations and thus millions of PDE solves to 
achieve meaningful Uncertainty Quantification (UQ). Since each PDE solve is computationally 
intensive, this introduces a severe computational bottleneck making their practical deployment for real-time imaging untenable, see e.g. \cite{afkham2024bayesian}.

To overcome this computational bottleneck, a two-stage hybrid-Bayesian 
framework for such inverse problems was recently proposed in \cite{koers2024}. This approach takes the following form: one first considers the linear problem of estimating some auxiliary parameter that depends on the forward data from the observations. This is done by formulating a non-parametric Bayesian regression problem for the said linear problem. Under the considered noise model and suitable priors (such as a 
Gaussian process prior), the posterior distribution for the auxiliary parameter 
given the observations is often available in closed form, thereby 
avoiding MCMC-based sampling entirely. In the second stage, this 
posterior distribution (or, measure) is pushed forward to the parameter space through 
a deterministic non-linear map drawn from the classical inverse problems 
literature. The resulting pushforward posterior can then be viewed 
as a Bayesian solution to the inverse problem, under the corresponding pushforward prior. In fact, a basic measure-theoretic observation, proved in Section \ref{sec: method}, shows that the
pushforward of the auxiliary posterior is absolutely continuous with respect
to the pushforward prior and therefore defines the posterior associated with
the induced prior on the parameter space. Conceptually related ideas based on pushforwarding measures through deterministic maps have been used in several other contexts, such as ensemble Kalman inversion \cite{iglesias_11} and projection-posterior methods, \cite{ghosal25}.

However, extending the framework introduced in \cite{koers2024} to specific inverse problems is not straightforward and requires careful problem-specific analysis. For QPAT, one needs to identify a suitable auxiliary variable and establish the required properties of the corresponding reconstruction map. For EIT, the auxiliary variable is the Dirichlet-to-Neumann (DtN) operator instead of a function. This requires constructing computationally tractable priors on the space of DtN operators. Moreover, since explicit inversion formulas are not available, the conductivity reconstruction relies on iterative algorithms or numerical reconstruction algorithms leading to additional computational challenges. 
\paragraph{Contributions.} In summary, the main contributions of this article are as 
follows:
\begin{enumerate}
    \item \textbf{Extension to QPAT and EIT.} We develop the  framework of \cite{koers2024} for QPAT and EIT.  
  The EIT case is particularly nontrivial because the auxiliary variable is the Dirichlet-to-Neumann (DtN) operator rather than a function. This requires constructing computationally tractable priors on operator spaces together with the corresponding posterior analysis, extending the function-space framework of \cite{koers2024}.

    \item \textbf{Theoretical analysis.} We establish the contraction rates for the induced pushforward posterior in both QPAT and EIT. While contraction rates for the direct Bayesian 
    posterior in related elliptic inverse problems have been studied 
    \cite{NiAb19,afkham2024bayesian}, to the best of our knowledge,  our analysis is the first to address 
    contraction rates for a pushforward posterior in this setting. 

    \item \textbf{Computational efficiency and uncertainty quantification.} 
    We provide, to the best of our knowledge, the first pushforward-based 
    Bayesian UQ framework for EIT and QPAT, achieving exact posterior sampling 
    without recourse to MCMC\cite{cotter2013mcmc} or QMC methods \cite{bazahica2025uncertainty}, 
    at a fraction of the computational cost. Through numerical experiments, 
    we demonstrate that the proposed approach yields credible intervals 
    and posterior uncertainty maps for both QPAT and EIT while also being computationally efficient.
\end{enumerate}
The remainder of this article is organized as follows. In 
Section~\ref{sec: two_imaging_problems}, we introduce the two inverse 
problems studied in this article, QPAT and EIT, describing 
the governing PDE models, the associated forward maps, and the 
statistical observation models for each. 
Section~\ref{sec: method} extends the two-stage hybrid-Bayesian 
methodology of \cite{koers2024} as applied to both problems. For 
each problem, we derive posterior contraction rates, and construct 
credible regions for the reconstructed parameters. The two problems 
are treated in parallel, illustrating two distinct regimes of the framework: inversion from internal data in QPAT 
(Section~\ref{subsec:qpat_method}) and inversion from boundary data in EIT(Section~\ref{subsec:eit_method}). Numerical experiments for both problems are reported in 
Section~\ref{sec:numerics}, where we demonstrate the computational 
efficiency and uncertainty quantification capabilities of the 
proposed approach. We conclude with a discussion in 
Section~\ref{sec: conclusion}.

\section{The two imaging inverse problems} \label{sec: two_imaging_problems}
In this section, we introduce the two non-linear inverse problems 
studied in this article: QPAT
and EIT. 
For each problem, we describe 
the governing PDE model, the associated forward map, and the 
statistical observation model that forms the basis of the two-stage 
hybrid-Bayesian framework employed in this article. QPAT is treated 
in Section~\ref{subsec: QPAT_description} and EIT in Section~\ref{subsec: EIT_description}.

\subsection{Quantitative Photo-Acoustic Tomography (QPAT)}\label{subsec: QPAT_description}
Quantitative photoacoustic tomography is a hybrid imaging method in which optical excitation and acoustic measurements are combined to reconstruct internal optical parameters of the medium, \cite{ammari_book,bal_review,Tarvainen2024QPAT}. In this modality a short optical pulse illuminates the medium and part of the optical energy is absorbed inside the tissue. This absorption produces a small thermoelastic expansion and generates acoustic waves that are measured on the boundary. From these acoustic data, one first reconstructs the absorbed energy distribution in the interior. The second step is the quantitative problem (QPAT) where one seeks to recover the optical parameters of the medium from this internal (reconstructed) data. In strongly scattering media the propagation of light is commonly described by the diffusion approximation which is the macroscopic model of light transport. In this article, we work under this diffusion approximation for QPAT and describe the resulting model, below.

Let $D\subset\mathbb{R}^d$, $d=2,3$, be a smooth bounded domain. Under the diffusion approximation, the PDE model for quantitative photoacoustic tomography (QPAT) is given by
\begin{equation}
-\nabla\cdot(\mu\nabla u)+\gamma u=0 \quad \text{in } D,
\qquad
u=g \quad \text{on } \partial D,
\label{eq:PDE}
\end{equation}
where $\mu$ and $\gamma$ denote the diffusion and absorption coefficients, respectively, $u$ denotes the light field, and the prescribed Dirichlet boundary function $g$ models the incoming optical illumination. Throughout this work, we assume that the diffusion coefficient $\mu\in C^{0,1}(D)$ is known, uniformly elliptic, and satisfies
$0<\frac{1}{M}\leq \mu\leq M<\infty.
$
The absorption coefficient $\gamma$, which is the parameter of interest, is assumed to satisfy
$\gamma\in H^s(D), s\ge1,
$
together with the pointwise bounds
$\frac{1}{m}\leq\gamma\leq m.
$
For convenience, we denote this admissible class by
$\gamma\in H_m^s(D).
$

The absorbed optical energy density is given by
$H:=\gamma u.
$
Under the above assumptions, if the prescribed boundary illumination satisfies $g\ge g_{\min}>0$ and $g\in H^{1/2}(\partial D)$, then standard elliptic theory implies that the boundary value problem admits a unique weak solution $u\in H^1(D)$, which is uniformly positive in $D$.

Here and throughout the paper, $H^s(X)$ denotes the Sobolev space of order $s$ on the domain $X$. The inverse problem in QPAT is to recover the optical parameters from measurements of the absorbed optical energy density $H$. In this article, we assume that the diffusion coefficient $\mu$ is known and fixed, and focus only on recovering the absorption coefficient $\gamma$. The corresponding forward map is therefore given by (see also \cite{afkham2024bayesian})
\[
G:\gamma\mapsto H:=\gamma u,
\qquad
G:H_m^s(D)\rightarrow H^1(D).
\]

The statistical observation model is taken to be the Hilbert-space Gaussian white noise model
\begin{equation}\label{eq:noisyM}
Y=H+\varepsilon\xi,\qquad H:=G(\gamma),
\end{equation}
where $\xi$ denotes Gaussian white noise process. This observation model admits several asymptotically equivalent formulations. First, relative to an orthonormal basis $(e_k)_{k\ge1}$ of $H^1(D)$, it can be written as
\begin{align}
Y_k:=\langle Y,e_k\rangle
=\langle H,e_k\rangle+\varepsilon\xi_k,
\label{GSS_QPAt}
\end{align}
where $\xi_k\sim N(0,1)$ are independent standard Gaussian random variables.
Secondly, in a pointwise discrete observational setting, one observes noisy pointwise evaluations of the absorbed energy density at design points $\{x_i\}_{i=1}^{N},\ N\to \infty$,
\begin{align}
Y_i=H(x_i)+\varepsilon W_i,
\label{discrete_QPAT}
\end{align}
where $W_1,\ldots,W_N$ are independent standard normal random variables.

Throughout the remainder of the paper, we use the abstract Gaussian white noise model \eqref{eq:noisyM} as a convenient notation for these equivalent formulations; see also \cite{afkham2024bayesian,Knapik2011Bayesian}. Since the forward map $H=G(\gamma)$ depends nonlinearly on the unknown absorption coefficient $\gamma$, recovering $\gamma$ from the noisy observations $Y$ constitutes a nonlinear statistical inverse problem. The proposed methodology for solving this problem is described in Section~\ref{sec: method}.

\subsection {Electrical Impedance Tomography (EIT)} \label{subsec: EIT_description}
Electrical impedance tomography is a non-invasive imaging modality in which electrical currents and voltages measured on the boundary of a body are used to infer the electrical conductivity distribution inside the medium. In a typical EIT experiment electrodes are attached to the boundary of the medium and small currents are injected while the resulting voltages are recorded. These measurements contain information about the internal conductivity of the medium (e.g., different tissues exhibit different electrical conductivities) and can be used in applications such as lung monitoring, stroke detection, and breast imaging \cite{Agnelli_2020,chere_02,frer_02}.

We now describe the mathematical formulation of the EIT forward and 
inverse problems.
Let $D \subset \mathbb{R}^d$ be a bounded domain with smooth boundary and let $\sigma(x)>0$ denotes the conductivity. We model the unknown conductivity as belonging to the following 
admissible class: 
\begin{align}\Sigma^{\alpha}_{m,D_0}=\left\{\sigma \in C(D) : \inf_{x\in D}\sigma(x) \ge m, \quad\sigma = 1 \text{ on } D \setminus D_0 \text { and }\norm{\sigma}_{H^{\alpha}}\leq {\frac{1}{m}}
\right\},
\end{align}
where $m \in (0,1)$ and $D_0$ is a domain compactly contained in $D$. In other words, the conductivity is strictly positive and differs from the known background value only inside $D_0$. In absence of any current source or sinks within the medium, the elliptic PDE that models the electric potential within the medium is
\begin{align}\label{eq:eit_pde}
\nabla \cdot (\sigma \nabla u) = 0 \quad \text{in } D.
\end{align}
To define the boundary operator associated with the conductivity $\sigma$, we consider the Dirichlet boundary value problem for the PDE \eqref{eq:eit_pde} with the associated Dirichlet data $u=f$ on $\partial D$ where $f\in H^{s+1}_{\diamond}(\partial D)$ belongs to some appropriate Sobolev class with integrals vanishing on the boundary. For $\sigma$ and $f$ as above it is well-known \cite{NiAb19,hank_11}, that the Dirichelt BVP admits a unique weak solution $u\in H^{\min\{1,s+3/2\}}(D)$ in an appropriate class of functions. The corresponding Neumann boundary data $\frac{\partial u_{\sigma,f}}{\partial \nu}\Big|_{\partial D}$
belongs to the space
$H^s_{\diamond}(\partial D) := \{ g \in H^s(\partial D) : \langle g,1\rangle_{L^2(\partial D)} = 0 \}.
$
This allows one to define the Dirichlet-to-Neumann operator
\[
\Lambda_\sigma : H^{s+1}(\partial D)/\mathbb{C} \to H^s_{\diamond}(\partial D),
\qquad
f \mapsto \frac{\partial u_{\sigma,f}}{\partial \nu}\Big|_{\partial D}.
\]
In practice currents are prescribed on the boundary and voltages are measured, corresponding to the Neumann-to-Dirichlet (N-t-D) operator, while theoretical analysis typically uses the Dirichlet-to-Neumann (D-t-N) map as above
which maps prescribed boundary voltages to the resulting current flux; the two formulations provide mathematically equivalent descriptions of the same measurements. The inverse problem of recovering $\sigma$ from a knowledge of this boundary operator (D-t-N) is known in the literature as the Calder\'on problem\cite{cald_80}, which is a highly non-linear ill-posed inverse problem. Following, \cite{NiAb19}, we consider noisy observations of the shifted Dirichlet-to-Neumann operator
$\tilde{\Lambda}_\sigma = \Lambda_\sigma - \Lambda_1,
$
where $\Lambda_1$ denotes the operator corresponding to the background conductivity $\sigma \equiv 1$. This centering removes the known background response and yields $\tilde{\Lambda}_\sigma$ that lies in a (Hilbert) space of Hilbert-Schmidt operators, allowing the observation model to be written as the Gaussian white noise model
\begin{align}\label{eq:noisyE}
Y = \tilde{\Lambda}_\sigma + \varepsilon W,
\end{align}
where $W$ is Gaussian white noise and $\varepsilon>0$ denotes the noise level. 
As in the QPAT setting, this observation model admits several asymptotically equivalent formulations, see \cite[Appendix D]{NiAb19}. First, in a practically motivated electrode measurement model, one observes
\begin{align}
Y_{p,q}
=
\langle
\widetilde{\Lambda}_\sigma[\phi_p],
\phi_q
\rangle_{L^2(\partial D)}
+
\varepsilon g_{p,q},
\qquad
p,q=1,\ldots,P,
\label{discrete:CEM}
\end{align}
where $\{\phi_p\}$ are functions supported on the electrode locations and
$g_{p,q}\sim N(0,1)$ are independent standard Gaussian random variables.

Alternatively, one may consider spectral measurements relative to an orthonormal basis $(\varphi_k)_{k\ge1}$ of $L^2(\partial D)$, leading to observations of the form
\begin{align}
Y_{j,k}
=
\langle
\widetilde{\Lambda}_\sigma[\varphi_j],
\varphi_k
\rangle_{L^2(\partial D)}
+
\varepsilon g_{j,k},
\label{GSS_EIT}
\end{align}
where again $g_{j,k}\sim N(0,1)$ are independent standard Gaussian random variables.

Throughout the remainder of the paper, we use the abstract Gaussian white noise model \eqref{eq:noisyE} as a convenient notation for these equivalent formulations. The statistical inverse problem is to recover the conductivity
$\sigma\in\Sigma^{\alpha}_{m,D_0}$
from noisy observations of the form \eqref{eq:noisyE}. We describe the proposed inversion methodology in the next section.




\section{Linear Method for Nonlinear Problems}
\label{sec: method}
We begin this section by mentioning that inversion methods for both of the problems considered here are available in both classical (deterministic)  setting as well in a Bayesian setting, see e.g. \cite{Abhi_20,NiAb19,afkham2024bayesian,arridge_12,mg_20,mg_21,jin2012reconstruction,kaipio_00,tanja_19}. Furthermore, to speed up computations in MCMC-based reconstruction methods for similar inverse problems, model-reduction based approaches (e.g., \cite{cui2016scalable,cui2015data}) and surrogate modeling for the forward operator (see \cite{li2014adaptive,yan2017convergence}) have also been proposed which speed up likelihood evaluations in MCMC samplers. Parallel to these works, modern operator-learning based approaches have also been successfully deployed for use in such MCMC-based Bayesian inversion methods as surrogate-models for the forward operator, see e.g. \cite{gao2023adaptive,majee_25,zhou2020adaptive}.

As described in section \ref{intro}, the two-stage hybrid Bayesian framework as proposed in \cite{koers2024} takes a different route compared to the usual method adopted in the literature on Bayesian inverse problems which begins by placing a prior on the parameter of interest and then implementing MCMC based approach to explore the posterior so obtained. In the new framework, the Bayesian update is performed on an auxiliary linear problem, while the actual parameter of interest is obtained by deterministic pushforward of the posterior measure obtained after the Bayesian update.
The analysis for QPAT and EIT is carried out in Sections~\ref{subsec:qpat_method} and~\ref{subsec:eit_method} respectively. Before moving forward however, we will establish a basic measure theoretic result showing that pushforward of the auxiliary posterior is the induced  measure corresponding to the pushforward of the prior measure of the auxiliary parameter which is absolutely continuous with respect to the prior measure, see Figures \ref{fig:1} and \ref{fig:2}. Theorem \ref{thm:1} below forms the starting point of our analysis. This theorem is not novel and forms the central rationale for proving posterior consistency results for various non-linear Bayesian inverse problems, see \cite[section 2.2]{afkham2024bayesian} and, \cite[section 5]{NiAb19}. In our work however, following the philosophy of \cite{koers2024}, we use this not just for proving theoretical consistency results for the posterior but rather to devise an algorithmic framework to obtain a posterior on the respective parameter of interest for a given non-linear inverse problem.

\begin{theorem}\label{thm:1}
Let $X$ and $\Gamma$ be separable Banach spaces, and let
$e:X\to\Gamma$
be a Borel measurable map. Furthermore, let $e$ be injective and $e^{-1}:e(X)\to X$ is also Borel measurable. Let $\mu_0$ be a Borel probability measure on $X$, and define the pushforward prior
$\Pi_0:=e_{\#}\mu_0,
$
i.e.,
$\Pi_0(B)=\mu_0(e^{-1}(B)),
\forall
B\in\mathcal B(\Gamma).$
Suppose the data are generated according to the following linear observation model,
\[
Y=Ax+\varepsilon W,
\] where $A:X\to \mathcal{Y}$ is a linear map, $W$ is Gaussian white noise and $\varepsilon>0$ denotes the noise level.
Let us denote the corresponding negative log likelihood by $\varphi(x;y)$ where $\varphi:X \times \mathcal{Y}\to \mathbb{R}$ is a measurable function. Suppose the posterior measure on $X$ is given by
\[
\frac{d\mu^y}{d\mu_0}(x)
=
\frac{1}{Z(y)}\exp(-\varphi(x;y)),
\]
where
$Z(y)=\int_X \exp(-\varphi(x;y))\,\mu_0(dx)>0.
$ Consider the pushforward measure
$
\Pi^y:=e_{\#}\mu^y.
$
Then $\Pi^y$ is absolutely continuous with respect to $\Pi_0$, and
\[
\frac{d\Pi^y}{d\Pi_0}(\gamma)
=
\frac{1}{Z(y)}
\exp\bigl(-\varphi( e^{-1}(\gamma);y)\bigr)
\]
for $\Pi_0$-almost every $\gamma\in e(X)$. 
Equivalently, the pushforward measure on $\gamma$ is the Bayesian posterior associated with the pushforward prior $\Pi_0=e_{\#}\mu_0$ and the transformed forward map
$\widetilde G(\gamma)=Ae^{-1}(\gamma),\,
\gamma\in e(X).
$
\end{theorem}
\begin{proof}
Let $B\in\mathcal B(\Gamma)$. Observe that,
$\Pi^y(B):=\mu^y(e^{-1}(B)).
$
By the statement in the theorem,
we have the posterior measure $\mu^y$ is absolutely continuous with respect to the prior measure, $\mu_0$ i.e. $\mu^y\ll\mu_0$. We will now show that the pushforward of the posterior measure is absolutely continuous with respect to that of the prior measure. To that end consider $B\in \mathcal{B}(\Gamma)$ such that $\Pi_0(B)=0$, then clearly,
$\mu_0(e^{-1}(B))
=\Pi_0(B)=
0.$
Thus,
$\Pi^y(B)
=\mu^y(e^{-1}(B))
=0,
$ by the fact that $\mu^y\ll \mu_0$.
Therefore $\Pi^y\ll\Pi_0$.

Again using the Radon--Nikodym representation of $\mu^y$,
\[
\Pi^y(B)
=
\frac{1}{Z(y)}
\int_{e^{-1}(B)}
\exp(-\varphi(x;y))
\,\mu_0(dx)
=
\frac{1}{Z(y)}
\int_X
\mathbf 1_B(e(x))
\exp(-\varphi(x;y))
\,\mu_0(dx).
\]
Since $e$ is injective and $e^{-1}:e(X)\to X$ is Borel measurable, we have
$\int_X
F(e(x))
\,\mu_0(dx)
=
\int_\Gamma
F(\gamma)
\,\Pi_0(d\gamma)
$
for every bounded measurable function $F$. Choosing
\[
F(\gamma)
=
\mathbf 1_B(\gamma)
\exp\bigl(-\varphi(e^{-1}(\gamma);y)\bigr),
\]
we obtain
\[
\Pi^y(B)
=
\frac{1}{Z(y)}
\int_B
\exp\bigl(-\varphi(e^{-1}(\gamma);y)\bigr)
\,\Pi_0(d\gamma).
\]
Since this identity holds for every Borel set $B$, the Radon-Nikodym theorem yields
\[
\frac{d\Pi^y}{d\Pi_0}(\gamma)
=
\frac{1}{Z(y)}
\exp\bigl(-\varphi(e^{-1}(\gamma);y)\bigr)
\]
for $\Pi_0$-almost every $\gamma\in e(X)$. Finally, by defining
$\widetilde G(\gamma)
=Ae^{-1}(\gamma),
\,
\gamma\in e(X),
$
we prove that $\Pi^y$ is indeed the Bayesian posterior associated with the pushforward prior $\Pi_0=e_{\#}\mu_0$ and the transformed forward map $\widetilde G$.
\end{proof}

\subsection{Bayesian inversion in QPAT} \label{subsec:qpat_method}
Based on the strategy outlined in the introduction, we now give details of our two-step hybrid Bayesian inversion method for QPAT. The first step involves specifying a deterministic inversion map for $\gamma$ when we are given observations of the (non-noisy) data $H$ and specified boundary data $g$. To that end, we recall the following simple inversion method from \cite[Page 9, eq. (17)]{bal_11}. We outline the steps in detail in order to get Lipschitz estimates for the reconstruction map that will be needed for posterior analysis. Consider
\begin{equation}
L:=-\nabla\cdot(\mu\nabla\cdot)
\end{equation}
so that \eqref{eq:PDE} can be rewritten as $Lu=\gamma u \text{ in } D$. Let us denote $H:=\gamma u$. Thus, in a non-noisy set-up we can recover $u$ by solving the Dirichlet BVP,
\begin{align*}
Lu&=H \quad \text{in } D\\
u&=g \quad \text{on }\partial D.
\label{eq: PAT equation}
\end{align*} Note that from elliptic regularity for uniformly elliptic operators with Lipschitz coefficients $\mu$ as above and the boundary condition $g\geq g_{\min}>0$, there exists a solution $u\in H^1(D)$ which is positive, i.e. $u>0$. Once we have recovered $u$, next, we recover $\gamma$ pointwise by $\gamma=H/u$. We will exploit this simple pointwise inversion in our two-step process for the Bayesian recovery of $\gamma$ given noisy observed data, $Y$. To summarize this approach, we first define an auxiliary variable, $v:=Lu$, such that
\begin{equation}
v=H \quad \text{in } D.
\end{equation}
Then the above-mentioned pointwise recovery of the unknown parameter of interest $\gamma$ from $v$ amounts to carrying out the following process.


Let $K: L^2(D) \rightarrow H^2(D) \cap H^1_0 (D)$ be the Dirichlet inverse of $L$ and $\tilde g$ be a function such that:
\begin{align}
L(Kw)&=w,\quad Kw|_{\partial D}=0,\\
L\tilde g&=0,\quad \tilde g|_{\partial D}=g.
\end{align}
Then it is easy to check that $u=Kv+\tilde g$ solves \eqref{eq:PDE}. Thus the pointwise inversion map is given by:
\begin{equation}
\gamma = e(v) := \,\frac{v}{K v+\tilde g}.
\end{equation}
We briefly recall the two-step approach of \cite{koers2024}, illustrated in figure \ref{fig:1} below. A Bayesian posterior is first computed for the auxiliary parameter $v$ by solving a linear regression problem. The posterior for the parameter of interest $\gamma$ (corresponding to a pushforward prior) is then obtained by applying the deterministic inversion map $e$, i.e., $\Pi_\gamma(\cdot \mid Y)=e_\#\Pi_v(\cdot \mid Y)$. In the case, where exact sampling for the auxiliary parameter can be done (e.g. for an additive Gaussian model with Gaussian process priors), this framework allows one to bypass expensive MCMC based sampling to explore the posterior for the parameter of interest. 
\begin{figure}[ht]
\centering
\[
\begin{tikzcd}[row sep=5em, column sep=8em]
\Pi_v 
\arrow[r, "\text{Bayes' theorem}"{name=U, above}, 
       "\text{observe } Y"{below}]
\arrow[d, "e_\#"', "\text{pushforward}"{right}] 
& \Pi_v(\cdot \mid Y) 
\arrow[d, "e_\#", "\text{pushforward}"{left}] \\
\Pi_\gamma := e_\# \Pi_v 
\arrow[r, dashed, "\text{induced}"']
& \Pi_\gamma(\cdot \mid Y) := e_\#\Pi_v(\cdot \mid Y)
\end{tikzcd}
\]
\caption{Two-step Bayesian procedure: inference is first performed on the auxiliary variable $v$, and the posterior on $\gamma$ is then induced by pushforward through $e$. }
\label{fig:1}
\end{figure}

\subsubsection{Lipschitzness of the deterministic recovery map} In order to transfer the inference result from the auxiliary variable $v$ to the parameter of interest $\gamma$ in a principled manner, we need to examine the properties of the inversion map, $e(\cdot)$. 
{Recall the deterministic recovery map
\[
e(v) := \,\frac{v}{Kv+\tilde g}.
\]
In what follows we will work on subsets $V_\star \subset L^2(D)$ on which $e$ is well-defined and stable.
Specifically, we assume that for $v\in V_\star$ the corresponding field
$u(v):=Kv+\tilde g
$
satisfies a uniform positivity condition, i.e., there exists constants $m>0$, $M<\infty$, such that:
\begin{equation}\label{eq:u_pos}
\inf_{x\in D} u(v)(x) \ge m>0,
\qquad v\in V_\star,
\end{equation}
and that the induced absorption coefficient is uniformly bounded,
\begin{equation}\label{eq:gamma_bdd}
\sup_{x\in D} |e(v)(x)| \le M < \infty,
\qquad v\in V_\star.
\end{equation}
(These are natural in QPAT: $g\ge g_{\min}>0$ ensures $u>0$, and the physical constraint
$\gamma\in H^s_m(D)$ corresponds to uniform bounds.)}
\begin{lemma}\label{lem: lipschitz_gamma}
Let $D$ be a smooth Lipschitz domain and $\mu>0$ be the known diffusion parameter as above. Let $u_i = K v_i + \tilde g$ for $i=1,2$, with $u_i \ge m >0$ in $D$. 
Define $\gamma_i :=  v_i/u_i$. 
Then the map $v \mapsto \gamma$ is locally Lipschitz in $L^2(D)$:
$\|\gamma_1 - \gamma_2\|_{L^2}
\le C \|v_1 - v_2\|_{L^2}, \,C:=\frac{1+M\|K\|}{m}
$
where $M=\mathrm\sup_D \gamma_2$ and $\|K\|$ is the operator norm of the bounded map $K:L^2(D)\to H^{2}(D)\hookrightarrow L^2(D)$.
\end{lemma}
\begin{proof}
Consider
\begin{align*}
  \gamma_1 - \gamma_2
= \frac{v_1}{u_1} - \frac{v_2}{u_2}
= \frac{v_1 - v_2}{u_1} + v_2 \left( \frac{1}{u_1} - \frac{1}{u_2} \right).  
\end{align*}
Using
\begin{align*}
\frac{1}{u_1} - \frac{1}{u_2} = \frac{u_2 - u_1}{u_1 u_2},
\quad u_2 - u_1 = K(v_2 - v_1),
\end{align*} we obtain
\begin{align*}
\gamma_1 - \gamma_2
= \frac{v_1 - v_2}{u_1} - \frac{v_2}{u_1 u_2} K(v_1 - v_2).
\end{align*}
Taking norms yields
\[
\|\gamma_1 - \gamma_2\|_{L^2}
\le \frac{\|I - \gamma_2 K\|}{m}\,\|v_1 - v_2\|_{L^2}.
\]
Finally, since $\|I - \gamma_2 K\| \le 1 + M\|K\|$, the result follows as the operator norm of $K$ is bounded, see e.g.\cite[chapter 6]{evans_book}.
\end{proof}

\subsubsection{Posterior evaluation}\label{qpat:posterior}
As stated above, in the QPAT setting the auxiliary variable is the absorbed energy density $v := H=\gamma u$, so that the observation model becomes the linear Gaussian model
\[
Y = v+\varepsilon \xi .
\] where this abstract model may be understood in the sense of the discrete model given by \eqref{GSS_QPAt} or, \eqref{discrete_QPAT}. Let us place a Mat\'ern Gaussian prior on $v\sim \Pi_v := \mathcal{N}(0,C_{\nu,\ell}),$ where \(C_{\nu,\ell}\) denotes the Mat\'ern covariance operator with smoothness parameter
\(\nu\) and length-scale \(\ell\). By \cite[Proposition 3.1]{Knapik2011Bayesian},  the posterior distribution of \(v\) given \(Y\) is Gaussian,
\[
\Pi_v(\cdot \mid Y)=
{N}(m_{v\mid Y},C_{v\mid Y}),
\]  where, $m_{v\mid Y}
=
C_{\nu,\ell}
\left(\varepsilon^2 I+C_{\nu,\ell}\right)^{-1}Y,$ and $C_{v\mid Y}
=C_{\nu,\ell}-C_{\nu,\ell}
\left(\varepsilon^2 I+C_{\nu,\ell}\right)^{-1}
C_{\nu,\ell}.$
\begin{proposition}[Posterior contraction result for QPAT]
Let $D \subset \mathbb{R}^d$, $d\in\{2,3\}$, be a bounded $C^{1,1}$ domain. 
Let $\mu \in C^{0,1}(\overline D)$ be uniformly elliptic and let $g \ge g_{\min}>0$. 
Define
\[
L := -\nabla\cdot(\mu\nabla\cdot), 
\qquad 
K : L^2(D) \to H^2(D)\cap H^1_0(D)\hookrightarrow L^2(D)
\]
to be the Dirichlet inverse of $L$, and let $\tilde g$ denote the $L$--harmonic extension of $g$.

For $v\in L^2(D)$, define
\[
u(v) := Kv + \tilde g,
\qquad
e(v) := \frac{v}{u(v)} .
\]

Assume the linear white--noise observation model
\[
Y = v + \varepsilon \xi, 
\]
where $v=\gamma u$ and the abstract model above is to be understood in the sense of \eqref{discrete_QPAT}. 
Let $\Pi_v$ be a suitable prior for $v$ on admissible subsets of $L^2(D)$, and define the induced (pushforward) prior for $\gamma$ as
$\Pi_\gamma := e_\# \Pi_v =  \Pi_v \circ e^{-1}.$ We will also denote the corresponding posteriors by $\Pi_{v}(\cdot|Y)$ and $\Pi_{\gamma}(\cdot|Y):=\Pi_v(\cdot|Y)\circ e^{-1}.$ respectively.

Suppose there exist measurable sets $V_\epsilon \subset L^2(D)$ and a sequence $\delta_\epsilon\downarrow 0$ such that:
\begin{enumerate}
\item[(A1)] for every $M_\epsilon \to \infty$,
$
\Pi_v\left( \|v-v_0\|_{L^2} > M_\epsilon \delta_\epsilon \middle| Y \right)
\xrightarrow{P_{v_0}} 0,
$
\item[(A2)] $\Pi_v(V_\epsilon \mid Y) \xrightarrow{P_{v_0}} 1$,
\item[(A3)] the map $e : V_\epsilon \to L^2(D)$ is locally Lipschitz on $V_\epsilon$. 
\end{enumerate}

Then, under (A1), (A2), and (A3) for every $\widehat M_\epsilon \to \infty$, 
\[
\Pi_\gamma\left( \|\gamma-\gamma_0\|_{L^2} > \widehat M_\epsilon \delta_\epsilon \middle| Y \right)
\xrightarrow{P_{\gamma_0}} 0 .
\]
\end{proposition}

\begin{proof}

{First of all, note that Assumption (A1) is satisfied for Gaussian priors under a linear Gaussian observation model from standard Bayesian non-parametric theory, see e.g. \cite[section 3]{Knapik2011Bayesian}. Now we establish the following relation between posterior measures for $\gamma$ to posterior measure of for $v$.}
Recall that $\gamma=e(v)$ and $v_0:=\gamma_0 u_0$. Hence,
$\gamma-\gamma_0 = e(v)-e(v_0).
$
Fix $\epsilon$ and let $v\in V_\epsilon$. 
From Lemma \ref{lem: lipschitz_gamma}
it follows,
\begin{equation}\label{eq:transfer_event_inclusion}
\Bigl\{ v\in V_\epsilon:\ \|v-v_0\|_{L^2}\le r \Bigr\}
\subset
\Bigl\{ v\in V_\epsilon:\ \|e(v)-e(v_0)\|_{L^2}\le C r \Bigr\}. 
\end{equation}

Let $\widehat M_\epsilon\to\infty$ be arbitrary. Define $M_\epsilon := \widehat M_\epsilon/C$; then $M_\epsilon\to\infty$.
Consider the following tail event for $\gamma$,
$A_\epsilon := \Bigl\{\gamma:\ \|\gamma-\gamma_0\|_{L^2}> \widehat M_\epsilon\delta_\epsilon\Bigr\}.
$
From the definition of the pushforward measure it is easy to see that,
\begin{align}
\Pi_{\gamma}(A_\epsilon\mid Y)
= \Pi_{v}(e^{-1}A_\epsilon|Y) \label{eq: pushforward}
\end{align}
Note that,
$e^{-1}A_\epsilon=\{v:\|e(v)-e(v_0)\|_{L^2}>\widehat M_\epsilon\delta_\epsilon\}
\subset
\{v:\|v-v_0\|_{L^2}>\frac{\widehat M_\epsilon}{C}\delta_\epsilon
= M_\epsilon\delta_\epsilon\}.
$
Furthermore, using the fact that
$e^{-1}(A_\epsilon)=
\bigl(e^{-1}(A_\epsilon)\cap V_\epsilon\bigr)\cup \bigl(e^{-1}(A_\epsilon)\cap V_\epsilon^c\bigr),
$
and from \eqref{eq: pushforward} we get,
\begin{align}
\Pi_\gamma(A_\epsilon \mid Y)
&=\Pi_v\bigl(e^{-1}(A_\epsilon)\mid Y\bigr)=
\Pi_v\bigl(e^{-1}(A_\epsilon)\cap V_\epsilon \mid Y\bigr)
+
\Pi_v\bigl(e^{-1}(A_\epsilon)\cap V_\epsilon^c \mid Y_\epsilon\bigr) \notag \\
&\leq
\Pi_v(V_\epsilon^c \mid Y)+\Pi_v\bigl(e^{-1}(A_\epsilon)\cap V_\epsilon \mid Y\bigr)
\notag\\
&\leq \Pi_v(v\notin V_\epsilon\mid Y_\epsilon)+\Pi_v\left(v \in V_\epsilon:\ \|v-v_0\|_{L^2}>M_\epsilon\delta_\epsilon \middle| Y\right).
\label{eq:split}
\end{align}
By assumptions (A1) and (A2) in the proposition, both terms converge to $0$ in $P_{v_0}$-probability, hence
\[
\Pi_{\gamma}\left(\gamma: \|\gamma-\gamma_0\|_{L^2}>\widehat M_\epsilon \delta_\epsilon \middle| Y\right)\to^{P_{\gamma_0}} 0,
\]
which proves the desired contraction result.
\end{proof}

\subsubsection{Preimage credible regions}
\label{subsec:uq_vdv}

We now describe a construction of credible regions for $\gamma$
that follows the general framework of \cite{koers2024}.
This approach emphasizes exact transfer of posterior credibility and
frequentist coverage from the linear parameter $v$ to the nonlinear
parameter $\gamma$.
\begin{remark}[Credible region and coveregae probability in Hilbert Space]
    In this work, we will need to work with Hilbert spaces of functions as well as that of operators. In order to talk about credible regions in the different Hilbert spaces used in various contexts in this paper, we will use the following general measure theoretic definition for Credible regions.
Let $H$ be a Hilbert space and $\Pi(\cdot|Y)$ be a posterior probability measure on the Borel $\sigma$-algebra of $H$. Then any measurable set, $C(Y)$ satisfying $\Pi(C(Y)|Y)= 1-\alpha$
 is called a $(1-\alpha)$-credible region. {From a frequentist point of view, the true parameter is fixed and the randomness comes from the noisy data. Thus, if \(v_0\in H\) denotes the true value of the parameter and the data are generated according to the model with truth \(v_0\), we write \(P_{v_0}\) for the corresponding probability law. The coverage of the credible region \(C(Y)\) at \(v_0\) is defined by
$P_{v_0}\bigl(C(Y)\ni v_0\bigr).$}
  
\end{remark}
\noindent Recall that in our two-step formulation the auxiliary parameter $v$ and
the parameter of interest $\gamma$ are related deterministically by
$
\gamma = e(v)
\text{ and } v_0 = e^{-1}(\gamma_0),
$
on the admissible set where the recovery map $e$ is well-defined and
one-to-one. Let $\widetilde C^v_\epsilon(Y)\subset L^2(D)$ be any posterior
credible region for $v$ with credible level $1-\alpha$, i.e.
\begin{equation}
\Pi_v\big(\widetilde C^v_\epsilon(Y)\mid Y\big)=1-\alpha.
\label{eq:vdv_cred_v}
\end{equation}
Here, subscript $\epsilon$ on $\widetilde{C}$ denotes the noise scaling and superscript denotes the unknown parameter whose credible region is constructed.

\paragraph{Preimage credible region for $\gamma$.}
Following \cite{koers2024}, we define the induced credible region
for $\gamma$ by the image of $\widetilde C^v_\epsilon(Y)$ under the recovery map:
\begin{equation}
C^{\gamma}_\epsilon(Y)
:= e\big(\widetilde C^v_\epsilon(Y)\big)
= \{\gamma:\ \gamma=e(v)\ \text{for some } v\in \widetilde C^v_\epsilon(Y)\}.
\label{eq:vdv_preimage}
\end{equation}

\begin{theorem} \label{thm: 3.4}
Recall that by construction of the two-stage hybrid-Bayesian framework, 
the posterior distribution of $\gamma$ given $Y$ is the pushforward of 
the posterior distribution of $v$ under the deterministic map $e$, i.e., 
$\Pi_\gamma(\cdot \mid Y) = \Pi_v(\cdot \mid Y) \circ e^{-1},
$
where $e$ is well-defined and one-to-one on the support of 
$\Pi_v(\cdot \mid Y)$. Let $\widetilde{C}_\epsilon^v(Y) \subset L^2(D)$ be 
a $(1-\alpha)$ posterior credible set for $v$ in the sense of the 
Definition in \eqref{eq:vdv_cred_v}, and define the induced credible set for $\gamma$ by
$
C_\epsilon^\gamma(Y) := e\bigl(\widetilde{C}_\epsilon^v(Y)\bigr)
$ as in \eqref{eq:vdv_preimage}.
Then the following hold:
\begin{enumerate}
    \item[\textnormal{(i)}] \textbf{Exact posterior credibility:}
$\Pi_\gamma(C_\epsilon^\gamma(Y) \mid Y) = \Pi_v(\widetilde{C}_\epsilon^v(Y) \mid Y).$
    \item[\textnormal{(ii)}] \textbf{Exact coverage transfer:}
    For $\gamma_0 = e(v_0)$, $
    P_{\gamma_0}(\gamma_0 \in C_\epsilon^\gamma(Y)) 
    = P_{v_0}(v_0 \in \widetilde{C}_\epsilon^v(Y)).$
    \item[\textnormal{(iii)}] \textbf{Diameter control:} For any subset $A\subset L^2(D)$ define $\operatorname{diam}_{L^2}(A):=\sup\{\lVert f-g\rVert_{L^2}:f,g\in A\}$. On the event 
    $\widetilde{C}_\epsilon^v(Y) \subset V_\epsilon$,
    $\operatorname{diam}_{L^2}(C_\epsilon^\gamma(Y)) 
    \leq C\, \operatorname{diam}_{L^2}(\widetilde{C}_\epsilon^v(Y)).
    $
\end{enumerate}

\end{theorem}
\begin{proof}
Recall the pushforward relation,
$
\Pi_\gamma(A \mid Y) = \Pi_v(e^{-1}(A) \mid Y)
$
for any measurable set $A \subset L^2(D)$. Substituting, $A = C_\epsilon^\gamma(Y)$ and using the definition of $C_\epsilon^\gamma(Y)$, we get
$
e^{-1}(C_\epsilon^\gamma(Y)) = \widetilde C_\epsilon^v(Y),
$
which proves (i).\\
Now consider the events, $\{\gamma_0 \in C_\epsilon^\gamma(Y)\}$ and $\{v_0 \in \widetilde C_\epsilon^v(Y)\}$. We will show they are equivalent. To that end, note that by definition:
$\gamma_0 \in C_\epsilon^\gamma(Y)
 \iff \gamma_0 = e(v) \text{ for some } v \in \widetilde C_\epsilon^v(Y).$ Since $e$ is one to one and $\gamma_0 = e(v_0)$, this holds if and only if $v = v_0$.
 From this (ii) follows.\\
By definition and from (ii) above, we have
$\operatorname{diam}_{L^2}(C_\epsilon^\gamma(Y))
=
\sup_{v_1, v_2 \in \widetilde C_\epsilon^v(Y)} \| e(v_1) - e(v_2) \|_{L^2(D)}.
$
Using the fact that $e$ is Lipschitz and by taking the supremum over $v_1, v_2 \in \widetilde C_\epsilon^v(Y)$ we get (iii).


\end{proof}

\subsection{Bayesian Inversion in EIT} \label{subsec:eit_method}
We now turn to the second inverse problem considered in this article. 
Let $\Lambda_{\sigma}$ be the DtN operator and $\Lambda_1$ be the known background operator. Consider now the shifted DtN operator, $\tilde{\Lambda}_\sigma:=\Lambda_\sigma-\Lambda_1$. The shifted DtN operator belongs in the space of Hilbert-Schmidt operators. For the reader's convenience, we briefly recall the definition of the Hilbert--Schmidt operator space from \cite{NiAb19}. Let $
\{ \phi_k := \phi_k^{(0)} \}_{k=0}^{\infty}$ be an orthonormal basis of $L^2(\partial\Omega)$ consisting of real-valued Laplace-Beltrami eigenfunctions. Removing the constant eigenfunction $\phi_0$ yields orthonormal basis of mean-free space of functions, \(L^2_\diamond(\partial\Omega)\). By an appropriate rescaling of these eigenfunctions, $\{ \phi_k^{(r)}  \}_{k=1}^{\infty} $ one can obtain an orthonormal basis
of all \( H^r(\partial \Omega)/\mathbb{C} \)  In this work, we assume $r=0,1$. Now, define the  tensor-product operators, \( b_{jk}^{(r)} : H^r(\partial \Omega) \to L^2(\partial \Omega) \) 
\[
b_{jk}^{(r)}(f) = \phi_j^{(r)} \otimes \phi_k^{(0)} (f) := \langle f, \phi_j^{(r)} \rangle_{H^r(\partial D)} \phi_k^{(0)}, \quad f \in H^r(\partial \Omega),
\]
and define the space of linear operators
\[
\mathbb{H}_r := \left\{ T : H^r(\partial \Omega) \to L^2(\partial \Omega), \quad T = \sum_{j,k=1}^\infty t_{jk} b_{jk}^{(r)} : t_{jk} \in \mathbb{R}, \sum_{j,k=1}^\infty t_{jk}^2 < \infty \right\}.
\]
The elements of \( \mathbb{H}_r \) are the ‘Hilbert–Schmidt’ operators between \( H^r(\partial \Omega) \) and \( L^2(\partial \Omega) \). Furthermore, this space is
equipped with the Hilbert--Schmidt inner product
\[
\langle S, T \rangle_{\mathbb{H}_r} := \sum_{j,k=1}^\infty s_{jk} t_{jk} \equiv \sum_{j,k=1}^\infty \langle S \phi_j^{(r)}, \phi_k^{(0)} \rangle_{L^2(\partial D)} \langle T \phi_j^{(r)}, \phi_k^{(0)} \rangle_{L^2(\partial D)}.
\] From \cite [Lemma 18]{NiAb19}, we know that $\tilde{\Lambda}_{\gamma}\in \Hb_r$ for any $r$. 
Recall the observation model \eqref{eq:noisyE} introduced in 
Section~\ref{subsec: EIT_description} \cite[eq. 12]{NiAb19}:
\begin{align}\label{eq: EIT_model}
Y=\tilde{\Lambda}_{\sigma}+\varepsilon W, 
\end{align}
where, $\Lambda_{\sigma}-\Lambda_1:=\tilde{\Lambda}_{\sigma}$ is the `shifted' D-t-N operator. Recall that as per \cite{NiAb19}, observing this model amounts to writing:
\[
Y(T) = \langle \tilde{\Lambda}_\sigma, T \rangle_{\mathbb{H}^r} + \varepsilon W(T),
\] for some test element $T\in \mathbb{H}_r$ and where $W$ is a centered Gaussian process satisfying
\[
\mathbb{E}[W(T)] = 0,
\qquad
\operatorname{Cov}(W(T_1), W(T_2)) = \langle T_1, T_2 \rangle_{\mathbb{H}^r}.
\] Carrying on with the same program as above, we will break our problem again into two parts: We will define a noisy linear observation model for the forward data (shifted Dirichlet-to-Neumann operator). In this step, we will do a Bayesian update to get the posterior on shifted D-t-N operators. Subsequently, we will push forward this posterior distribution to obtain a posterior distribution on the conductivity. In order to begin our discussion, we will need to choose a prior on the space to which such shifted DtN operators belong. Unlike the case in QPAT, this task of choosing a reasonable prior measure for the space to which DtN operators belong is highly non-trivial. This is because, the space to which the DtN operators belong is not a `function space'. We will first outline a way to construct priors on this space and subsequently explain our inversion pipeline. This construction follows closely the ideas in \cite[section 4.1] {Calvetti_2019}.
\subsubsection{Prior on the space of shifted DtN operators} \label{sec:priorDtN} As already stated, $\tilde{\Lambda}_{\sigma}\in \mathbb{H}_r$ is a (Hilbert) space of Hilbert-Schmidt operators. Unlike the QPAT setting, where a Gaussian prior can be specified directly on the auxiliary function, the auxiliary variable in EIT is an operator. The following  construction provides a computationally tractable Gaussian prior on this operator space. We start by defining first a Gaussian prior for the admissible parameters (log conductivities). The idea from \cite{Calvetti_2019} is to then generate the samples of the shifted DtN operators by evaluating the DtN operators corresponding to the sampled conductivities. This is equivalent to specifying a pushforward prior on the space of shifted DtN operators itself by pushforwarding a (Gaussian) measure on the space of conductivities to the space of shifted DtN operators. Subsequently a PCA type analysis is done to identify a Gaussian prior on the space of shifted DtN operators whose (empirical) mean and covariance are the same as the (empirical) first two moments of the pushforward prior on the space of DtN operators obtained as above. Once this Gaussian prior on the space of shifted DtN operators is specified, the Bayesian update and subsequent pushforward analysis proceed exactly as in the general framework. To put this construction on a concrete footing, we begin by recalling the following lemma that shows the existence of a probability measure with bounded second moments on the space of shifted DtN operators.

\begin{lemma}\label{thm:finite_moment}
    \cite[section 6.3]{castro24}Let us recall \eqref{eq:noisyE} and rewrite  $Y=\tilde{\Lambda}_\sigma+\epsilon W:=G(\sigma)+\epsilon W.$ If $\eta$ is a probability measure with bounded second moment (in particular, a Gaussian measure) on the space of admissible conductivities, then it's pushforward $G_{\sharp}\eta=\eta \circ G^{-1}$ is a probability  measure on the space of admissible operators with bounded second moment as well.
\end{lemma}
Note that the lemma \ref{thm:finite_moment} requires the conductivities to be drawn from a probability measure with bounded second moments. Examples of such probability measures include Gaussian measures which we have used in this project as well as, say, a probability measure over the space of bounded conductivities. This requirement of finite second moments is thus very natural for such inverse problems. As a simple consequence of the lemma above, we have the following theorem.
\begin{theorem} \label{thm:finite_covar}
    The covariance operator associated with the law of the random variables $\tilde{\Lambda}_{\sigma}$ is trace-class.
\end{theorem}

\begin{remark}We mention that the statement of the theorem \ref{thm:finite_covar} already appears as a comment in \cite{castro24} in a slightly different setting. For completeness' sake, we give a self-contained proof of the above statement. 
\end{remark}
\begin{proof}
In order to do the proof, we recall the abstract set-up from \cite[section 2.6]{gine_nickl}. Clearly, $\mathbb{H}_r$ is a seperable Hilbert space and the map $G:\Gamma \to \mathbb{H}_r$ given by $G(\sigma)=\tilde{\Lambda}_\sigma$ is measurable. Thus it follows from \cite[Lemma 2.6.3]{gine_nickl} that $\bar{\Lambda}=E(\tilde{\Lambda}_\sigma):=\int_{\mathbb{H}_r}{\tilde{\Lambda}_{\sigma}}d(\eta \circ G^{-1})=\int_\Gamma \tilde{\Lambda}_{\sigma}\,d\eta(\sigma)$ exists as a Bochner integral, since \begin{align}
     \int_\Gamma \norm{\tilde{\Lambda}_\sigma}_{\mathbb{H}_r}d\eta(\sigma)\leq \left(\int_\Gamma\norm{\tilde{\Lambda}_\sigma}_{\mathbb{H}_r}^2 d\eta(\sigma)\right)^{1/2}< \infty.
\end{align}
where the second integral is finite by Lemma \ref{thm:finite_moment}. Now let, $\Xig := \Gg - \Lambar \in \Hr$.
For $h, k \in \Hr$ define the bilinear form
\begin{align}
    B(h,k)
    := \int_{\Hr}\ip{\Xig}{h}_{\Hr}
      \ip{\Xig}{k}_{\Hr} d(\eta\circ G^{-1})=
    \int_\Gamma
      \ip{\Xig}{h}_{\Hr}
      \ip{\Xig}{k}_{\Hr}
    d\eta(\sigma).
\end{align}
Again by the Cauchy-Schwarz inequality,
$|\ip{\Xig}{h}_{\Hr}\ip{\Xig}{k}_{\Hr}|
\leq \norm{\Xig}_{\Hr}^2\norm{h}_{\Hr}\norm{k}_{\Hr}$. Since,
\begin{align}\label{eq:Xi-bound}
    \int_\Gamma \norm{\Xig}_{\Hr}^2\,d\eta(\sigma)
    \leq
    2\int_\Gamma \norm{\Gg}_{\Hr}^2 d\eta(\sigma)
    + 2\norm{\Lambar}_{\Hr}^2
    < \infty,
\end{align}
the bilinear form $B(\cdot,\cdot)$ is well defined and bounded. In particular,
\begin{align}
|B(h,k)|\leq\left(\int_\Gamma\norm{\Xig}_{\Hr}^2\,d\eta(\sigma)\right) \norm{h}_{\Hr}\norm{k}_{\Hr}.
\end{align} It is also easy to see that $B$ is self-adjoint (in fact, symmetric) and coercive. Then, as a consequence of the Reisz Representation Theorem (see e.g., \cite[Theorem 4.3.6]{debnath_book}), there exists a unique bounded
linear covariance operator $\mathcal{C} : \Hr \to \Hr$ satisfying 
 \begin{equation}\label{eq:cov-def}
              \ip{\mathcal{C}h}{k}_{\Hr}
              =\int_\Gamma\ip{\Gg - \Lambar}{h}_{\Hr}  \ip{\Gg -\Lambar}{k}_{\Hr} d\eta(\sigma),\qquad h, k \in\Hr.
          \end{equation}
Finally, we show that the covariance operator is trace-class. For this, let  $(\psi_j)_{j \geq 1}$ be any orthonormal basis of $\Hr$. Then, \begin{align}
    \operatorname{tr}(\mathcal{C})=\sum_{j=1}^{\infty}\ip{\mathcal{C}\psi_j}{\psi_j}_{\Hr} =\sum_{j=1}^{\infty} \int_\Gamma \ip{\Xig}{\psi_j}_{\Hr}^2 d\eta(\sigma).
\end{align}
Since, every term in the sum is nonnegative, so by the monotone convergence theorem we can interchange the order of the sum and the integral and write:
\begin{align}
     \operatorname{tr}(\mathcal{C})=\int_\Gamma \sum_{j=1}^{\infty} \ip{\Xig}{\psi_j}_{\Hr}^2 d\eta(\sigma)= \int_\Gamma\norm{\Xig}_{\Hr}^2\,d\eta(\sigma)< \infty.
\end{align} 
\end{proof}
As a consequence of the Lemma~\ref{thm:finite_moment} and 
Theorem~\ref{thm:finite_covar}, $\tilde{\Lambda}_\sigma$ admits a Karhunen-Lo\`eve (K-L) expansion
\begin{align}\label{KL_prior}
\tilde{\Lambda}_\sigma = \Lambar + \sum_{j=1}^\infty \alpha_j \psi_j,
\end{align}
where $(\psi_j)$ are eigenfunctions of $\mathcal{C}$ and $\alpha_j$ are uncorrelated coefficients. This gives us a justification for the following construction, which can be surmised as a Gaussian approximation to the above mentioned pushforward prior, see \cite[section 4.1]{Calvetti_2019}. \\
Let us sample conductivities $\sigma^{(k)}$ from a smoothness prior (such as one with some Mat\'ern kernel) and for each sample compute the corresponding shifted DtN operator. Let us represent the corresponding matrix using a Fourier basis, $\{\phi_j\}_{j=1}^{\infty}$ of $H^r(\Omega)$. Then, let us define the matrix form of the Dirichlet-to-Neumann map as the matrix with entries
\begin{equation}
(L_{\sigma})_{jk}
=
\int_{\partial \Omega} \phi_j \, (\tilde\Lambda_{\sigma} \phi_k)  dS
=
\langle \phi_j, \tilde\Lambda_{\sigma} \phi_k \rangle,
\quad 0 \le j, k < \infty.
\end{equation} In practice, we keep a truncated representation of such a matrix with $j,k\leq K\in \mathbb{N}$. Let us denote $d=(K)^2$. Each matrix is then vectorized to obtain vectors $w^{(k)} \in \mathbb{R}^{d}$.

Let $\bar w$ denote the empirical mean
$\bar w = \frac{1}{N}\sum_{k=1}^N w^{(k)}.
$
We then consider the centered samples
$\tilde w^{(k)} = w^{(k)} - \bar w.
$ Now suppose we draw $N$ samples and assemble the data matrix
\[
X =
\begin{bmatrix}
\tilde w^{(1)} & \tilde w^{(2)} & \cdots & \tilde w^{(N)}
\end{bmatrix}
\in \mathbb{R}^{d \times N},
\]
and compute its singular value decomposition
$X = U \Sigma V^{T}.
$
Note that the associated empirical covariance matrix is given by
\[
\widehat{C} = \frac{1}{N-1}\sum_{k=1}^N \tilde w^{(k)} (\tilde{w}^{(k)})^T
= \frac{1}{N-1}XX^T
= \frac{1}{N-1}U\Sigma \Sigma^T U^T.
\]

Now we can use the decay of the singular values in $\Sigma$ to construct a prior for the shifted DtN matrices in a lower dimensional subspace. The leading columns of $U$ provide a data-driven basis for this subspace. After we have the $R$ leading left singular vectors
\[
U = [ U_{1}, U_{2}, \dots, U_{R} ] \in \mathbb{R}^{d \times R}
\]
with associated variances
$\lambda_{j} = \frac{\Sigma_{jj}^{2}}{N-1},\, j = 1, \dots, R,$
and 
$\Lambda_{R} = \mathrm{diag}(\lambda_{1}, \dots, \lambda_{R}),
$
similar to \cite[Eq.~28]{Calvetti_2019}, we parameterize the random samples of the shifted DtN matrix (in a vectorized form) drawn from the (approximate) pushforward measure as described above, by:
\begin{align*}
w(\alpha) = &\bar w + \sum\limits_{j=1}^{R} U_{j}\alpha_j = \bar w + U\alpha,
\qquad
\alpha \sim {N}_R(0, \Lambda_{R})\\
&= \bar w + \sum_{j=1}^R \sqrt{\lambda_j}\xi_j U_j,
\qquad
\xi_j \sim {N}(0,1).
\end{align*}
In this way we construct a low-dimensional Gaussian approximation of the pushforward measure on the space of shifted DtN operators supported on the subspace spanned by the leading singular vectors. This will serve as our prior on the space of DtN operators.

\subsubsection{Posterior Evaluation}
Substituting this representation into the observation model yields
\[
Y(T) = \langle \bar w + \sum_{j=1}^R \sqrt{\lambda_j}\xi_j U_j, T \rangle_{\mathbb{H}}
+ \varepsilon W(T).
\]
Choosing ($T = U_j$), and using orthonormality of the vectors {$U_j$}, we obtain the following analogue of a Gaussian sequence space model (see also \cite[section 1]{NiAb19})
\begin{align}
Y_k := Y(U_k) = m_k + \sqrt{\lambda_k} \xi_k + \varepsilon Z_k,
\qquad k = 1, \dots, R
\end{align}
where
$
m_k = \langle \bar w, U_k \rangle_{\mathbb{H}_r},
$
and $Z_k := W(U_k)$. Since $W$ is white noise, it follows that $Z_k {\sim} {N}(0,1).$ From this it easily follows that the posterior distribution of each coefficient ($\xi_k$) is also a Gaussian distribution given by,
$\xi_k \mid Y_k \sim {N}(\mu_k, \sigma_k^2),$
with
\[
\sigma_k^2 = \left(1 + \frac{\lambda_k}{\varepsilon^2}\right)^{-1},
\quad
\mu_k = \sigma_k^2 \cdot \frac{\sqrt{\lambda_k}}{\varepsilon^2}(Y_k - m_k).
\]
Equivalently, we can write the posterior distribution for $\alpha_k = \sqrt{\lambda_k}\xi_k$ as:
\[
\alpha_k \mid Y_k \sim \mathcal{N}\left(
\frac{\lambda_k}{\lambda_k + \varepsilon^2}(Y_k - m_k),
\frac{\lambda_k \varepsilon^2}{\lambda_k + \varepsilon^2}
\right).
\] Finally, the posterior distribution of the shifted DtN operator is given by
\begin{equation}\label{formula:postDtN}
\widetilde{\Lambda}_\sigma| Y =
\bar w + \sum_{j=1}^R \alpha_j U_j,
\qquad
\alpha_j \mid Y \sim {N}\left(
\frac{\lambda_j}{\lambda_j + \varepsilon^2}(Y_j - m_j),
\frac{\lambda_j \varepsilon^2}{\lambda_j + \varepsilon^2}
\right),
\end{equation}

The next step is to use a deterministic solver to obtain a push-forward probability measure for the space of conductivities. In absence of analytical reconstruction formulas, in the literature iterative schemes \cite{baku93,hohage_97,kaltenbacher_08} as well as direct reconstruction schemes \cite{Abhi_20,arridge_12,mueller_03} are both popular. 


The map $\mathcal{R} : \mathbb{H}_r \to \Sigma^\alpha_{m,D_0}$
denotes a deterministic EIT reconstruction operator acting on individual
Dirichlet-to-Neumann operators. 
For a probability measure $\mu$ on $\mathbb{H}_r$, the pushforward
$\mathcal{R}_\# \mu$ denotes the induced probability measure on
$\Sigma^\alpha_{m,D_0}$.

The two-stage inference procedure for EIT is summarized in the 
figure \ref{fig:2} below. The Bayesian update is performed 
analytically at the level of the KL coefficients $\{\alpha_j\}$, where the posterior is available in closed form under a Gaussian prior as computed above. This corresponds to getting a posterior on the 
shifted DtN operator $\tilde{\Lambda}_\sigma$ via the KL expansion. In the final stage, the posterior on 
$\tilde{\Lambda}_\sigma$ is pushed forward to the conductivity 
space via the deterministic reconstruction operator 
$\mathcal{R} : \mathbb{H}_r \to \Sigma^\alpha_{m,D_0}$. As in the QPAT case, the dashed arrows indicate that the 
posteriors on $\tilde{\Lambda}_\sigma$ and $\sigma$ are 
\emph{induced} by the pushforward rather than obtained by direct 
Bayesian updates at those levels.
\begin{figure}[ht]
\centering
\[
\begin{tikzcd}[row sep=5em, column sep=9em]
\Pi_{\tilde{\Lambda}_\sigma} 
\arrow[r, "\text{Bayes' theorem}"{above}]
\arrow[d, "\mathcal{R}_\#"', "\text{pushforward}"{right}]
&
\Pi_{\tilde{\Lambda}_{\sigma}}(\cdot \mid Y)
\arrow[d, "\mathcal{R}_\#", "\text{pushforward}"{left}]
\\
\Pi_\sigma := \mathcal{R}_\# \Pi_{\tilde{\Lambda}_\sigma}
\arrow[r, dashed, "\text{induced}"']
&
\Pi_\sigma(\cdot \mid Y) := \mathcal{R}_\# \Pi_{\tilde{\Lambda}_\sigma}(\cdot \mid Y)
\end{tikzcd}
\]
\caption{Two-step Bayesian procedure: inference is first performed on the auxiliary variable $\tilde{\Lambda}_{\sigma}$, and the posterior on $\sigma$ is then induced by pushforward through $\mathcal{R}$. }
\label{fig:2}
\end{figure}
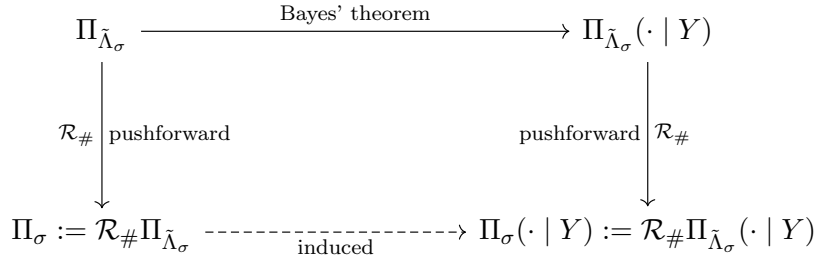

\begin{theorem}[Posterior contraction for EIT] \label{thm: post_cont} Consider the space of Hilbert-Schmidt operators, $\mathbb{H}_r, r\in \{0,1\}$ considered above and let $\tilde{\Lambda}_0:=\tilde{\Lambda}_{\sigma_0}\in \mathbb{H}_r.$ Let
$P_R:\mathbb{H}_r\to \operatorname{span}\{\psi_1,\dots,\psi_R\}
$
denote the orthogonal projection, and let us write
$
\tilde \Lambda_0-\bar{\Lambda}
=
\sum_{j=1}^r \alpha_{0j}\psi_j+E_R$,
where 
$E_R:=(I-P_r)(\tilde\Lambda_0-\bar{\Lambda}).
$ Let us assume that $\|E_R\|_{{\mathbb H}_r}\leq b_R$ such that $b_R\to 0.$ For the noisy observation model given by \eqref{eq: EIT_model}, we denote: For $j=1,\dots,R$, define
\[
Y_j:=\langle Y-\bar{\Lambda},\psi_j\rangle_{{\mathbb H}_r}.
\] In particular, this implies that $Y_j=\alpha_{0j}+\epsilon Z_j,
\
Z_j\sim N(0,1).$
Recall that inspired by the KL expansion given by \eqref{KL_prior}, we defined a Gaussian truncated prior of the form:
$
\tilde\Lambda
=
w+\sum_{j=1}^R \alpha_j\psi_j,
\
\alpha_j\sim N(0,\lambda_j),
$
with independent coefficients and $\lambda_j>0$ which are eigenvalues for the covariance operator with $\psi_j$ as the corresponding eigenvectors. The posterior on the coefficients, $\alpha_j$ was computed and satisfies $\alpha_j\mid Y_j\sim N(\mu_j,\tau_j^2)$, where $\mu_j=\frac{\lambda_j}{\lambda_j+\epsilon^2}Y_j,
\
\tau_j^2=\frac{\lambda_j\epsilon^2}{\lambda_j+\epsilon^2}.$ Let us assume, \begin{enumerate}
    \item[\textnormal{(A1)}] $\sum_{j=1}^R
\left(
\frac{\epsilon^2}{\lambda_j+\epsilon^2}
\right)^2
\alpha_{0j}^2
\leq C R\epsilon^2$ \text{ for some constant $C$ }
\item[\textnormal{(A2)}] there exist measurable sets $A_\epsilon\subset {\mathbb H}_r$ such that
$
\Pi_{\widetilde\Lambda}(A_\epsilon\mid Y)\to 1, 
$
\item[\textnormal{(A3)}] Reconstruction map $
\mathcal{R}:\widetilde\Lambda\mapsto \sigma
$
satisfies for all $\tilde\Lambda\in A_\epsilon$,
$
\| e(\widetilde\Lambda)-\sigma_0\|_{L^2(D)}
\leq
\omega\left(
\|\tilde\Lambda-\tilde\Lambda_0\|_{{\mathbb H}_r}
\right),
$
for some modulus of continuity $\omega$.
\end{enumerate}
Furthermore, define $\delta_{\epsilon,R}:=\epsilon\sqrt R+b_R.$ Then, for every $M_\epsilon\to\infty$,
\[
\Pi_\sigma
\left(
\|\sigma-\sigma_0\|_{L^2(D)}
>
\omega(M_\epsilon\delta_{\epsilon,R})
\mid Y
\right)
\overset{P_{\sigma_0}}{\rightarrow} 0
\]
\end{theorem}
\begin{proof}
 We first prove posterior contraction for the shifted DtN operator, $\tilde\Lambda$ in $\mathbb H_r$. By the Pythagoras theorem, it easily follows that:
 $\|\tilde\Lambda-\tilde\Lambda_0\|_{\mathbb H}^2
=
\sum_{j=1}^R(\alpha_j-\alpha_{0j})^2
+
\|E_R\|_{\mathbb H}^2.
$
By assumption, $\|E_R\|_{\mathbb H}\leq b_R.$ So, let us estimate, $\sum_{j=1}^r(\alpha_j-\alpha_{0j})^2.$ Note that,
\begin{equation}\label{eq:3.25}\sum_{j=1}^R(\alpha_j-\alpha_{0j})^2
\leq
2\sum_{j=1}^R(\alpha_j-\mu_j)^2
+
2\sum_{j=1}^R(\mu_j-\alpha_{0j})^2.
\end{equation} Recall, $\alpha_j-\mu_j|Y\sim N(0,\tau_j^2),\,\tau_j^2
= \frac{\lambda_j\epsilon^2}{\lambda_j+\epsilon^2}
\leq
\epsilon^2.$ This, $\mathbb{E}\left[
\sum_{j=1}^R(\alpha_j-\mu_j)^2
| Y\right]=\sum_{j=1}^R\tau_j^2
\leq R\epsilon^2.$ Then it follows from Markov's inequality that for every $M_\epsilon \to \infty$,
\begin{equation}\label{eq:3.24}
\Pi_{\alpha}
\left(
\sum_{j=1}^R(\alpha_j-\mu_j)^2
> M^2_\epsilon R\epsilon^2|Y
\right)
\leq
\frac{1}{M^2_\epsilon}\to 0.
\end{equation}


\noindent Now we  estimate the posterior mean error. By substituting, $Y_j=\alpha_{0j}+\epsilon Z_j,
\
Z_j\sim N(0,1),$ in the expression for $\mu_j$ above we get:


\[\sum_{j=1}^R(\mu_j-\alpha_{0j})^2
\leq 2\sum_{j=1}^R
\left(\frac{\epsilon^2}{\lambda_j+\epsilon^2}
\right)^2
\alpha_{0j}^2+2\epsilon^2\sum_{j=1}^R Z_j^2.
\] From Assumption (A1) in the theorem \ref{thm: post_cont} above and the fact that $\sum_{j=1}^R Z_j^2\sim \chi_R^2$ we get 
for every sequence \(M_\varepsilon\to\infty\),
\begin{equation}
\label{eq:posterior_mean_error}
P_{\tilde{\Lambda}_0}
\big(
\sum_{j=1}^R(\mu_j-\alpha_{0j})^2 >M_\varepsilon^2R\varepsilon^2
\big)
\to 0.
\end{equation}
Now it follows from \eqref{eq:3.25} that
\[
\{
\sum_{j=1}^R(\alpha_j-\tilde{\Lambda}_{0j})^2
>
4M_\varepsilon^2R\varepsilon^2
\}
\subseteq
\{
\sum_{j=1}^R(\alpha_j-\mu_j)^2 >M_\varepsilon^2R\varepsilon^2
\}
\cup
\{
\sum_{j=1}^R(\mu_j-\alpha_{0j})^2>M_\varepsilon^2R\varepsilon^2
\}.
\]
Conditioning on \(Y\), and noting that
\(\sum_{j=1}^R(\mu_j-\alpha_{0j})^2\) is fixed once \(Y\) is fixed, gives
\begin{align}
\Pi_\alpha
(
&\sum_{j=1}^R(\alpha_j-\alpha_{0j})^2
>
4M_\varepsilon^2R\varepsilon^2
\mid
Y
)\nonumber\\ 
&\leq
\Pi_\alpha
(
\sum_{j=1}^R(\alpha_j-\mu_j)^2>M_\varepsilon^2R\varepsilon^2
\mid
Y
)
+
\Pi_\alpha(\sum_{j=1}^R(\mu_j-\alpha_{0j})^2>M_\varepsilon^2R\varepsilon^2|Y)\nonumber\\
&\leq
\Pi_\alpha
(
\sum_{j=1}^R(\alpha_j-\mu_j)^2>M_\varepsilon^2R\varepsilon^2
\mid
Y
)
+
\mathbf 1_{\{\sum_{j=1}^R(\mu_j-\alpha_{0j})^2>M_\varepsilon^2R\varepsilon^2\}}.
\end{align}
Now from \eqref{eq:3.24} and \eqref{eq:posterior_mean_error}, we get \begin{equation}
\label{eq:coefficient_contraction}
\Pi_\alpha
(
\sum_{j=1}^R(\alpha_j-\alpha_{0j})^2
>
4M_\varepsilon^2R\varepsilon^2
\mid
Y
)
\overset{P_{\tilde{\Lambda}_0}}{\longrightarrow} 0
\end{equation}
We now transfer the coefficient contraction rate in
\eqref{eq:coefficient_contraction} to contraction of the shifted DtN
operator. Recall that
$
\tilde{\Lambda}
=
\bar w+\sum_{j=1}^R \alpha_j U_j,
$
whereas the fixed truth satisfies
$
\tilde{\Lambda}_0-\bar w
=
\sum_{j=1}^R \alpha_{0j}U_j+E_R,
\|E_R\|_{\mathbb H_r}\leq b_R .
$
Therefore
$
\tilde{\Lambda}-\tilde{\Lambda}_0
=
\sum_{j=1}^R(\alpha_j-\alpha_{0j})U_j-E_R .
$
By the triangle inequality and orthonormality of the vectors \(U_j\),
$
\|\tilde{\Lambda}-\tilde{\Lambda}_0\|_{\mathbb H_r}
\leq
(
\sum_{j=1}^R(\alpha_j-\alpha_{0j})^2
)^{1/2}
+
b_R .$
Define $\delta_{\varepsilon,R}:=2\varepsilon\sqrt R+ 2b_R.$
For $M_\varepsilon\geq 1$, if
$
\sum_{j=1}^R(\alpha_j-\alpha_{0j})^2
\leq
4 M_\varepsilon^2R\varepsilon^2,
$
then
$
\|\tilde{\Lambda}-\tilde{\Lambda}_0\|_{\mathbb H_r}
\leq
2 M_\varepsilon\varepsilon\sqrt R+ 2b_R
:=
M_\varepsilon\delta_{\varepsilon,R}. 
$
From this we get
$
\Pi_{\tilde{\Lambda}}
\left(
\|\tilde{\Lambda}-\tilde{\Lambda}_0\|_{\mathbb H_r}
>
M_\varepsilon\delta_{\varepsilon,R}
\,\middle|\,
Y
\right) 
\leq
\Pi_\alpha
(
\sum_{j=1}^R(\alpha_j-\alpha_{0j})^2
>
4 M_\varepsilon^2R\varepsilon^2
\mid
Y
).
$
Hence, 
\begin{equation}\label{eq:3.30}\Pi_{\tilde{\Lambda}}
\left(
\|\tilde{\Lambda}-\tilde{\Lambda}_0\|_{\mathbb H_r}
>
M_\varepsilon\delta_{\varepsilon,R}
\,\middle|\,
Y
\right)
\overset{P_{\tilde{\Lambda}_0}}{\rightarrow} 0 \text{ for every sequence}
M_\varepsilon\to\infty.
\end{equation}
Finally, to transfer the contraction from $\tilde{\Lambda}$ to $\sigma$, 
note that by Assumption~(A3), for any $\tilde{\Lambda} \in 
A_\varepsilon$,
\[
\|\tilde{\Lambda} - \tilde{\Lambda}_0\|_{\mathbb{H}_r} 
\leq M_\varepsilon \delta_{\varepsilon,r}
\implies
\|\mathcal{R}(\tilde{\Lambda}) - \sigma_0\|_{L^2(D)} 
\leq \omega(M_\varepsilon \delta_{\varepsilon,r}).
\]
Taking complements within $A_\varepsilon$ gives the set inclusion
\begin{align}\label{eq: set_inclusion_EIT}
\Bigl\{\tilde{\Lambda} \in A_\varepsilon : 
\|\mathcal{R}(\tilde{\Lambda}) - \sigma_0\|_{L^2(D)} > 
\omega(M_\varepsilon \delta_{\varepsilon,r})\Bigr\}
\subseteq
\Bigl\{\tilde{\Lambda} \in A_\varepsilon :
\|\tilde{\Lambda} - \tilde{\Lambda}_0\|_{\mathbb{H}_r} > 
M_\varepsilon \delta_{\varepsilon,r}\Bigr\}.
\end{align}
Let
$B_\varepsilon 
:= \left\{\sigma : \|\sigma - \sigma_0\|_{L^2(D)} > 
\omega(M_\varepsilon \delta_{\varepsilon,r})\right\}.
$
By definition,  
$
\Pi_\sigma(B_\varepsilon \mid Y) 
= \Pi_{\tilde{\Lambda}}\bigl(\mathcal{R}^{-1}(B_\varepsilon) \mid Y\bigr).
$
From \eqref{eq: set_inclusion_EIT},
\[
\mathcal{R}^{-1}(B_\varepsilon) \cap A_\varepsilon
\subseteq
\Bigl\{\tilde{\Lambda} \in A_\varepsilon :
\|\tilde{\Lambda} - \tilde{\Lambda}_0\|_{\mathbb{H}_r} > 
M_\varepsilon \delta_{\varepsilon,r}\Bigr\}.
\]
Splitting $\mathcal{R}^{-1}(B_\varepsilon)$ on $A_\varepsilon$ and 
$A_\varepsilon^c$ and using the above inclusion, we obtain
\begin{align*}
\Pi_\sigma(B_\varepsilon \mid Y)
&= \Pi_{\tilde{\Lambda}}\bigl(\mathcal{R}^{-1}(B_\varepsilon) \cap 
A_\varepsilon \mid Y\bigr) 
+ \Pi_{\tilde{\Lambda}}\bigl(\mathcal{R}^{-1}(B_\varepsilon) \cap 
A_\varepsilon^c \mid Y\bigr)\\
&\leq \Pi_{\tilde{\Lambda}}\!\left(
\tilde{\Lambda} \in A_\varepsilon :
\|\tilde{\Lambda} - \tilde{\Lambda}_0\|_{\mathbb{H}_r} > 
M_\varepsilon\delta_{\varepsilon,r} \mid Y\right)
+ \Pi_{\tilde{\Lambda}}(A_\varepsilon^c \mid Y).
\end{align*}
The first term converges to zero in \(P_{\tilde{\Lambda}_0}\)-probability
by \eqref{eq:3.30}. The second term converges to zero in
\(P_{\tilde{\Lambda}_0}\)-probability by Assumption \((A2)\). Hence
\[
\Pi_\sigma
\left(
\|\sigma-\sigma_0\|_{L^2(D)}
>
\omega(M_\varepsilon\delta_{\varepsilon,R})
\mid Y
\right)\overset{P_{\sigma_0}}{\longrightarrow}
 0.
\]

\end{proof}

\begin{theorem}[Credible region for EIT]\label{thm: credible_region_EIT} Let,  $\widetilde C_\epsilon^{\Lambda_{\sigma}}(Y)\subset \mathbb{H}_r$ be a $(1-\alpha)$ credible region for the posterior distribution of the operator, $\tilde\Lambda_\sigma$, i.e., $
\Pi_{\tilde\Lambda_\sigma}\bigl(\widetilde C_\epsilon^\Lambda(Y)\mid Y\bigr)
= 1-\alpha.$ Consider the pushforward posterior distribution of $\sigma$, $\Pi_{\widetilde\Lambda}(\cdot\mid Y)$, under the reconstruction map $\mathcal{R}:\widetilde\Lambda_\sigma \mapsto \sigma,$ and define
$C_\epsilon^\sigma(Y):=\mathcal{R}\bigl(\widetilde C_\epsilon^\Lambda(Y)\bigr).
$ Suppose that $\mathcal{R}$ satisfies a modulus of continuity estimate, $
|\mathcal{R}(\Lambda_1)-\mathcal{R}(\Lambda_2)|_{L^2(D)}
\leq
\omega\left(
|\Lambda_1-\Lambda_2|_{\mathbb{H}_r}
\right)
$
for all admissible $\Lambda_1,\Lambda_2$, where 
$\omega:[0,\infty)\to[0,\infty)$ is a monotone increasing modulus 
of continuity. Then we have the following:
\begin{enumerate}
    \item[\textnormal{(i)}] \textbf{Exact posterior credibility:}
$\Pi_\sigma\bigl(C_\epsilon^\sigma(Y) \mid Y\bigr) = 1 - \alpha.$
    \item[\textnormal{(ii)}] \textbf{Exact  coverage 
    transfer:} If $\sigma_0$ denotes the true conductivity and $\tilde\Lambda_0=\tilde\Lambda_{\sigma_0}$, then for $\sigma_0 = \mathcal{R}(\widetilde{\Lambda}_0)$, $P\bigl(\sigma_0 \in C_\epsilon^\sigma(Y)\bigr) 
    = P\bigl(\widetilde{\Lambda}_0 \in 
    \widetilde{C}_\epsilon^\Lambda(Y)\bigr).
    $
    \item[\textnormal{(iii)}] \textbf{Diameter control:}
    $\operatorname{diam}_{L^2(D)}\bigl(C_\epsilon^\sigma(Y)\bigr) 
    \leq \omega\!\left(\operatorname{diam}_{\mathbb{H}_r}
\bigl(\widetilde{C}_\epsilon^\Lambda(Y)\bigr)\right).
    $
\end{enumerate}

\end{theorem}
\begin{proof}
    The proof is identical to the proof of Theorem \ref{thm: 3.4} and we skip the details for brevity. We  note that proofs for $(i)$ and $(ii)$ follow simply from the credibility and coverage identities follow from the pushforward relation
$\Pi_\sigma(\cdot\mid Y)
=\Pi_{\widetilde\Lambda}(\cdot\mid Y)\circ \mathcal{R}^{-1}.
$ Finally, for  $(iii)$, let $\sigma_i=\mathcal{R}(\Lambda_i)$ with $\Lambda_i\in \widetilde C_\epsilon^\Lambda(Y)$ for $i=1,2$. Then
$|\sigma_1-\sigma_2|_X
=|\mathcal{R}(\Lambda_1)-\mathcal{R}(\Lambda_2)|_X
\le
\omega\left(
|\Lambda_1-\Lambda_2|_{\mathbb{H}_r}
\right).
$
Taking supremum over  $\Lambda_i\in \widetilde C_\epsilon^\Lambda(Y)$ and the fact that $\omega$ is monotone increasing gives the result.
\end{proof}

\begin{corollary}
\label{cor:credible_eit}
In the setting of Theorem~\ref{thm: credible_region_EIT}, suppose 
that $\mathcal{R}$ satisfies the logarithmic stability estimate of the 
Calder\'{o}n problem \cite[Lemma 7]{NiAb19}: there exist constants 
$C, \delta > 0$ such that for all admissible $\Lambda_1, \Lambda_2$ 
with $\|\Lambda_1 - \Lambda_2\|_{\mathbb{H}_r}$ sufficiently small,
$
\|\mathcal{R}(\Lambda_1) - \mathcal{R}(\Lambda_2)\|_{L^\infty(D)} 
\leq C\left|\log\|\Lambda_1 - \Lambda_2\|_{\mathbb{H}_r}
\right|^{-\delta}.
$
Assume further that 
$\operatorname{diam}_{\mathbb{H}_r}(\widetilde{C}_\epsilon^\Lambda(Y)) < 1$.
Then, from 
Theorem~\ref{thm: credible_region_EIT} we get,
\[
\operatorname{diam}_{L^2(D)}\bigl(C_\epsilon^\sigma(Y)\bigr)
\lesssim
\left|\log\operatorname{diam}_{\mathbb{H}_r}
\bigl(\widetilde{C}_\epsilon^\Lambda(Y)\bigr)\right|^{-\delta},
\]
where $\lesssim$ means inequality holds up to constants.
\end{corollary}
\begin{proof}
 We have $|\log t| > 0$ for all $t \in 
\bigl(0,\, \operatorname{diam}_{\mathbb{H}_r}
(\widetilde{C}_\epsilon^\Lambda(Y))\bigr]$, so the modulus of continuity 
$\omega(t) = C|D|^{1/2}|\log t|^{-\delta}$ is well-defined. 
Moreover, since $|\log t|$ is strictly decreasing on $(0,1)$, the 
function $t \mapsto |\log t|^{-\delta}$ is strictly increasing on 
$(0,1)$, so $\omega$ is monotone increasing. Furthermore, $
\|\mathcal{R}(\Lambda_1) - \mathcal{R}(\Lambda_2)\|_{L^2(D)} 
\leq |D|^{1/2}\|\mathcal{R}(\Lambda_1) - \mathcal{R}(\Lambda_2)\|_{L^\infty(D)}
\leq C|D|^{1/2}\left|\log\|\Lambda_1 - \Lambda_2\|_{\mathbb{H}_r}
\right|^{-\delta}
$. Setting 
$a := \operatorname{diam}_{\mathbb{H}_r}
(\widetilde{C}_\epsilon^\Lambda(Y)) < 1$, the supremum of $\omega(t)$ over 
$t \in (0, a]$ is attained at $t = a$, giving
\[
\sup_{t\,\in\,(0,\,a]}\, C|D|^{1/2}|\log t|^{-\delta} 
= C|D|^{1/2}|\log a|^{-\delta}.
\]
Thus we get the desired bound from applying Theorem \ref{thm: credible_region_EIT}.
\end{proof}

\section{Numerical Reconstruction}\label{sec:numerics}

In this section, we perform the numerical implementation of the proposed framework for reconstructing the parameter of interest (PoI) and quantify its uncertainty. Computationally, the power of the proposed uncertainty quantification framework lies in the fact that it separates Bayesian inference from the deterministic reconstruction method used in any particular experiment.  Bayesian inference is performed in the space of measurements, while the reconstruction algorithm is treated as a deterministic map from the data to the parameter space. Moreover, once the posterior samples of the relevant measurement are drawn, one can apply the deterministic reconstruction method on these samples in parallel to obtain the pushforward posterior samples of the PoI. This is another advantage over MCMC based methods where the MCMC iterations have to be performed sequentially. To elucidate these ideas, we implement the proposed framework for both QPAT and EIT. In particular, for EIT we use two different reconstruction methods, namely D-Bar based reconstruction \cite{silt_github,silt_20} and OOEIT \cite{Petri24, petri_github} for which we utilize their publicly available MATLAB packages on GitHub. To avoid inverse crime, forward data was generated on a finer mesh and the reconstructions carried on a coarser mesh in all cases. For brevity, we only report here the numerical experiments used to generate reconstructions below, but we carried out similar experiments across multiple noise levels. As expected, the reconstruction quality suffers as relative noise levels are increased, particularly for EIT.

\subsection{Implementation for QPAT}

To evaluate the performance of the proposed framework for uncertaint quantification in the QPAT setting, we consider a numerical experiment to estimate the absorption coefficient $\gamma(x)$ in \eqref{eq:PDE} from noisy internal data for absorbed optical energy density when $\mu(x)$ is assumed to be known. 
\subsubsection{Bayesian inference in measurement space}
Following the framework developed in Section \ref{subsec:qpat_method}, Bayesian inference is performed on the absorbed optical energy density $H$.  To begin our numerical experiments, simulated data for $H$ is generated on a fine mesh at $4\%$ noise wherein for the true absorption parameter $\gamma(x)$ we have considered a function with a constant background with two distinct circular inclusions inside the medium. The value within the inclusions is set to be $0.8$ and $1$ while the background conductivity is set at $0.2$. The diffusion coefficient $\mu(x)$ is assumed to be known and is obtained by smoothing a scaled version of the exact absorption coefficient. The procedure adopted here is the same as \cite[Section 6.3]{afkham2024bayesian}. The forward and inverse problems employ the same positive Dirichlet illumination, chosen as the sum of three Gaussian sources located near the boundary of the unit disk. Specifically, the prescribed boundary data is $g(x)=w_{m_1,s_1}(x)
+w_{m_2,s_2}(x)
+w_{m_3,s_3}(x)$ where $w_{m,s}(x)=s\exp\left(-2\|x-m\|^2\right),$ with $m_1=\frac{1}{2}(\sqrt{2},\sqrt{2}), \,
m_2=\frac{1}{2}(-\sqrt{2},\sqrt{2}), \,
m_3=-m_1,$ and $s_1=10,\,
s_2=2,\,
s_3=5.$
This choice ensures that the optical field remains strictly positive throughout the domain and is the same as in \cite[section6.3]{afkham2024bayesian}. Thereafter we use a coarser mesh for the reconstruction pipeline. A Gaussian prior with Matern covariance is placed on $H$, using smoothness parameter $\nu=1.5$ and correlation length $l=0.06$. Although, not necessarily required, the correlation length was selected using an empirical Bayes procedure and subsequently fixed throughout the numerical experiment.
After observing the noisy measurements, the posterior distribution of $H$ is computed in closed form as described in Section \ref{qpat:posterior}. As the posterior is also Gaussian, we can easily generate as many exact samples from it as needed. For illustration in our numerical experiments, we choose to draw $100$ posterior samples that were then pushed forward through the pointwise reconstruction map described in Section \ref{subsec:qpat_method}. Note that in typical MCMC based reconstruction methods for such high-dimesnional non-linear problems as considered in this work, the effective sample sizes (ESS) are typically around $100$ for MCMC chains of size $2.5\times 10^6$, see e.g  \cite{stuart_16}.
\subsubsection{Posterior uncertainty propagation}
The ensemble of pushed-forward absorption coefficient samples is used to characterize the posterior uncertainty in $\gamma$, see Figure \ref{fig:qpat-eit-mean}. The point estimate for $\gamma$ is given by the posterior mean which is computed as
$\bar{\gamma}
=
\frac{1}{100}
\sum_{k=1}^{100}
\gamma^{(k)},$ where $\gamma^{(k)}$ corresponds to the $k-$th pushed forward sample. Pixelwise posterior quantiles are also computed. In particular, the $2.5\%$ and $97.5\%$ quantiles are used to construct pointwise $95\%$ credible region. Finally, we also plot $95\%$ credible interval along a horizontal slice passing through the inclusions computed from the posterior ensemble. Additionally, we also visualize the credible interval for the absorbed optical energy density $H$ along the same horizontal slice. One can observe that width of the credible intervals for both $H$ and $\gamma$ along the same horizontal slice seems to be of the same order. This observation agrees with the transfer rate predicted by Theorem \ref{thm: 3.4}.

\begin{figure}[htbp]
\centering
\begin{tabular}{cccc}
\includegraphics[width=0.25\textwidth]{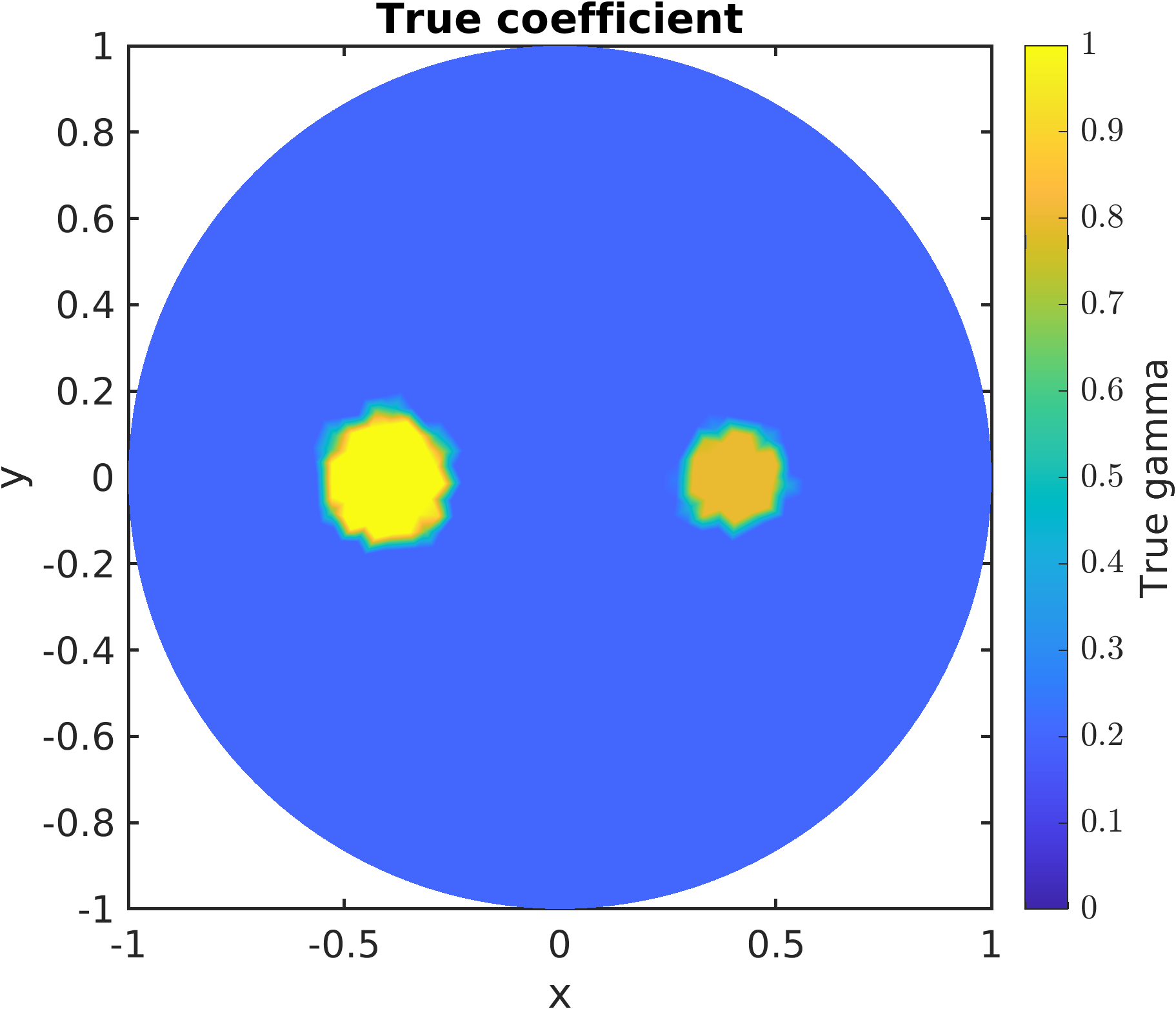}
&
\includegraphics[width=0.25\textwidth]{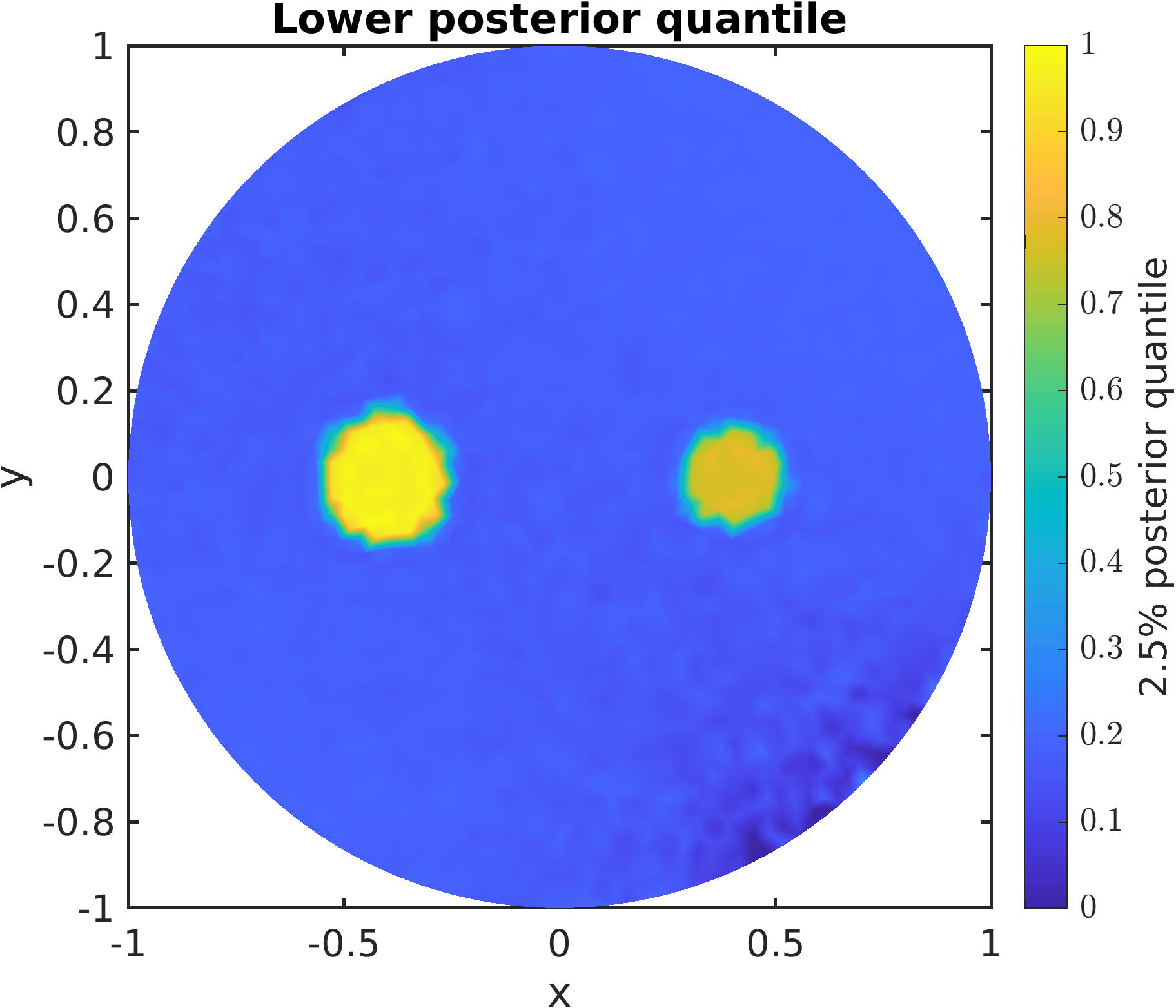}
&
\includegraphics[width=0.28\textwidth]{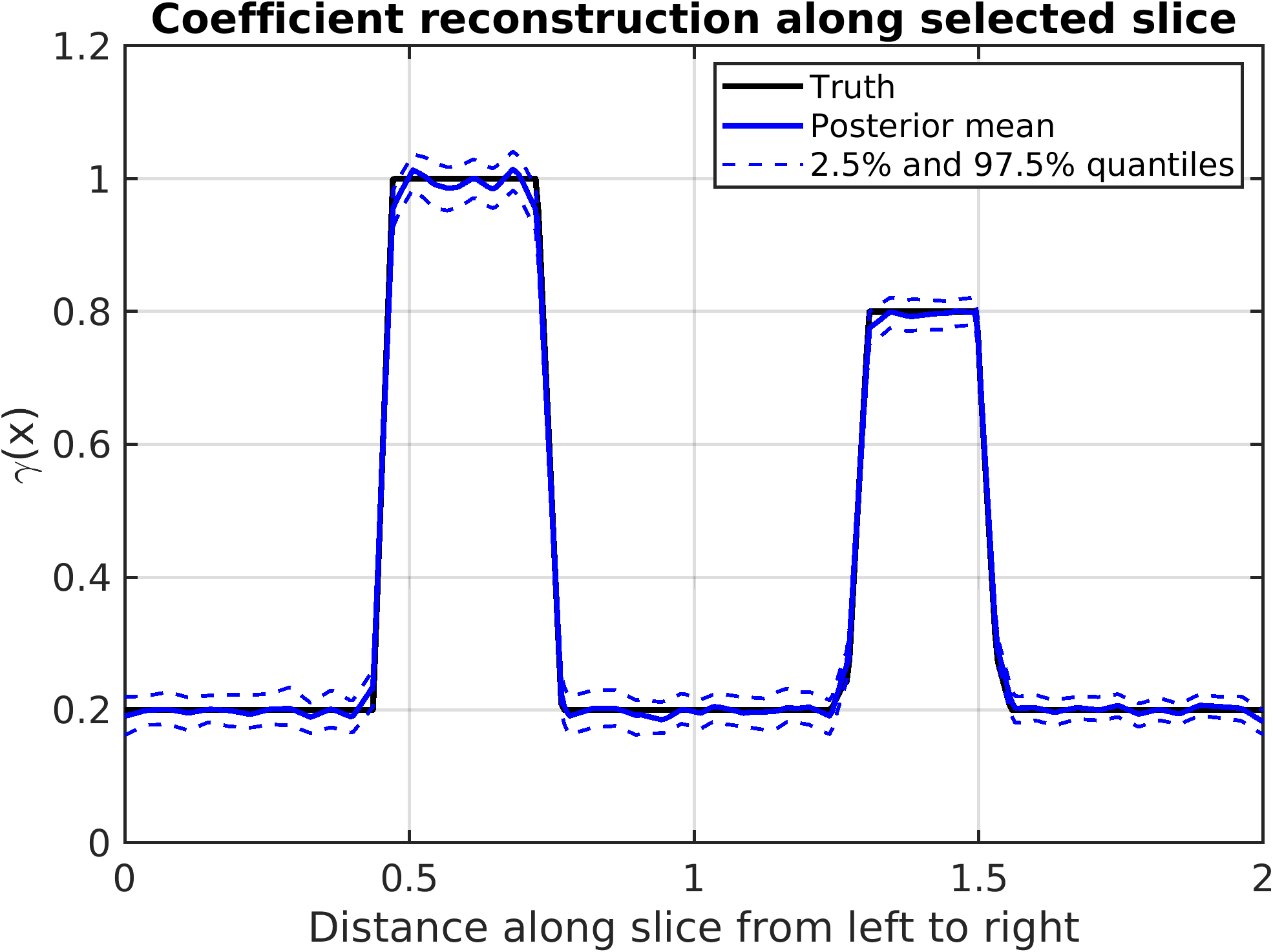}
\\
\includegraphics[width=0.25\textwidth]{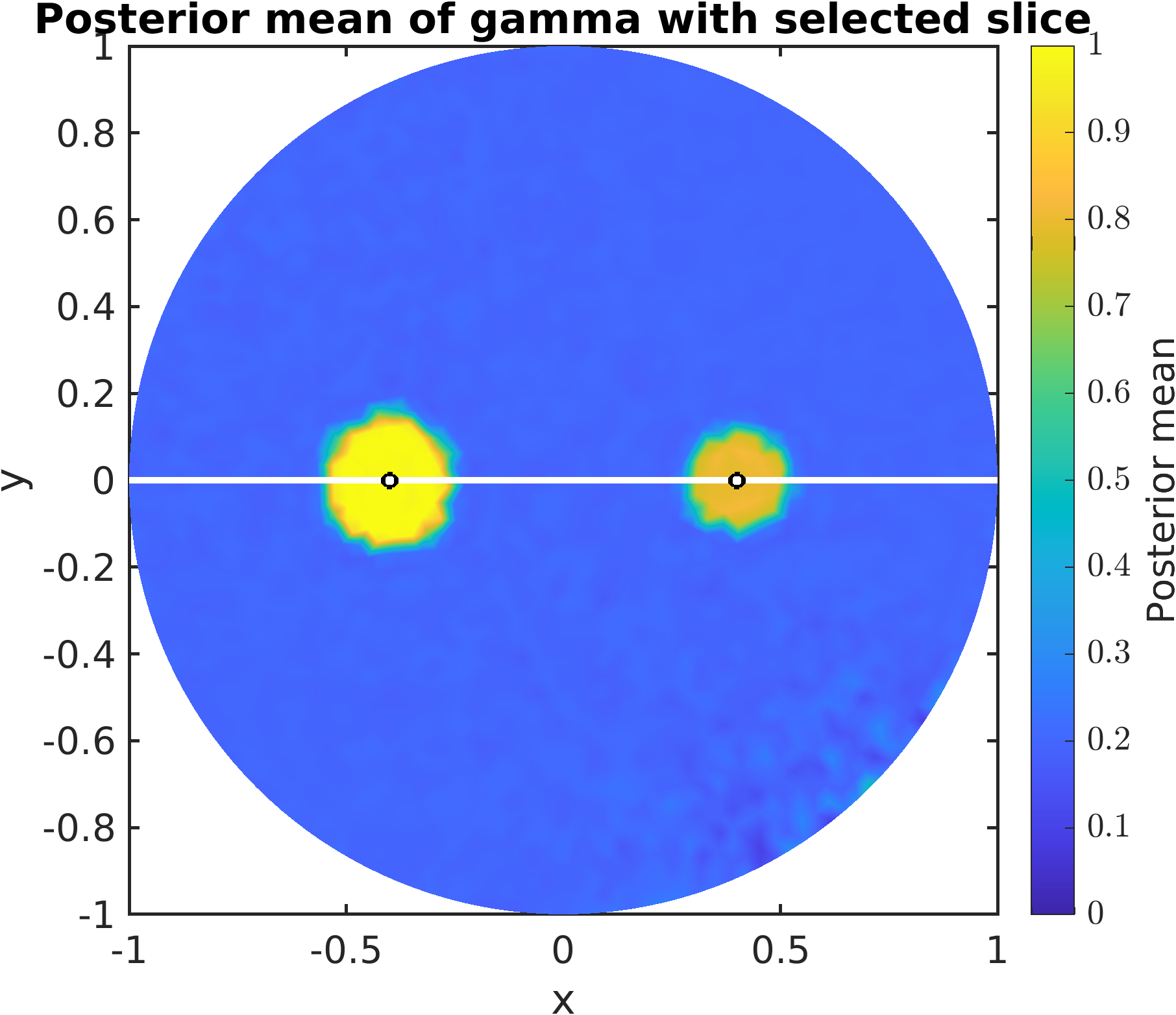}
&
\includegraphics[width=0.25\textwidth]{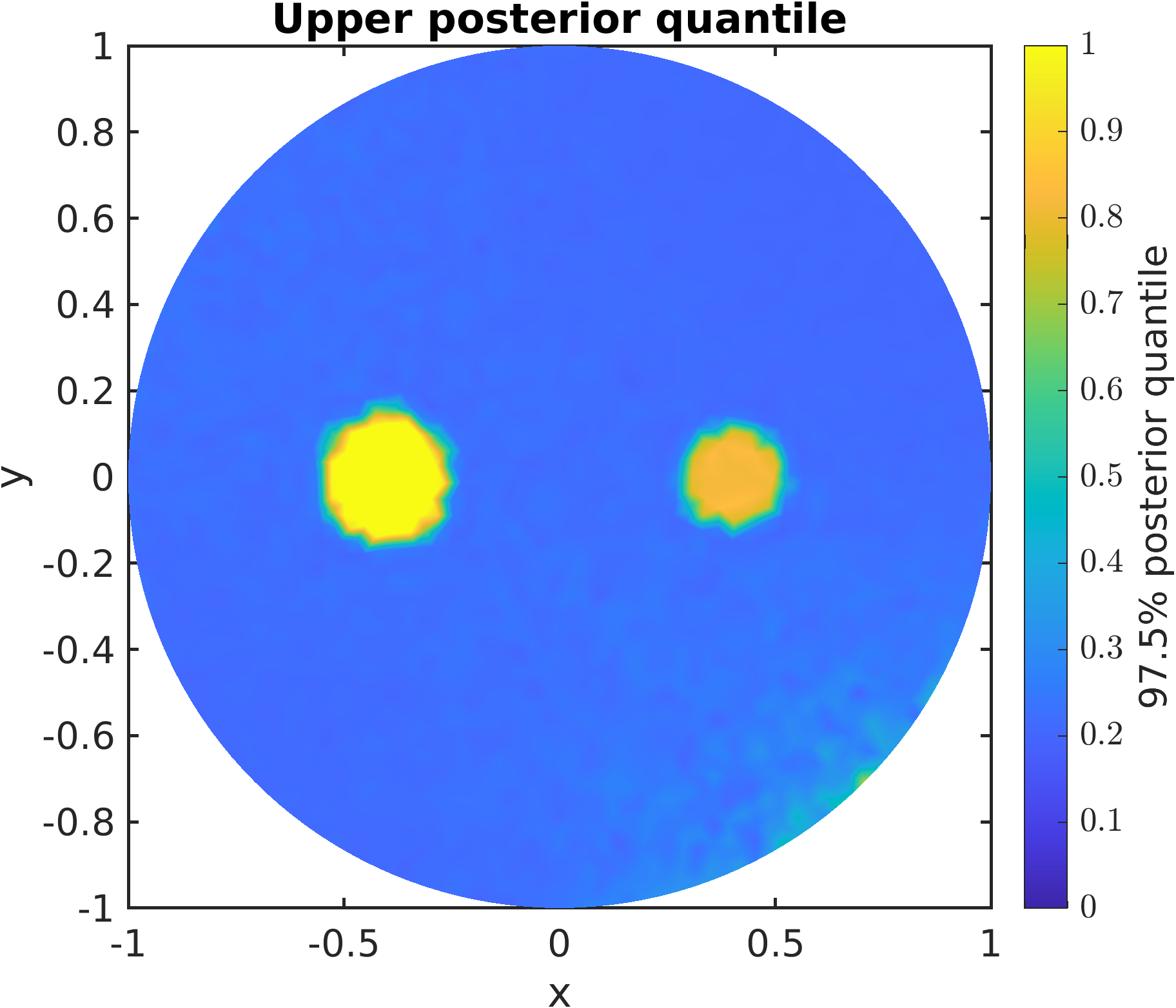}
&
\includegraphics[width=0.28\textwidth]{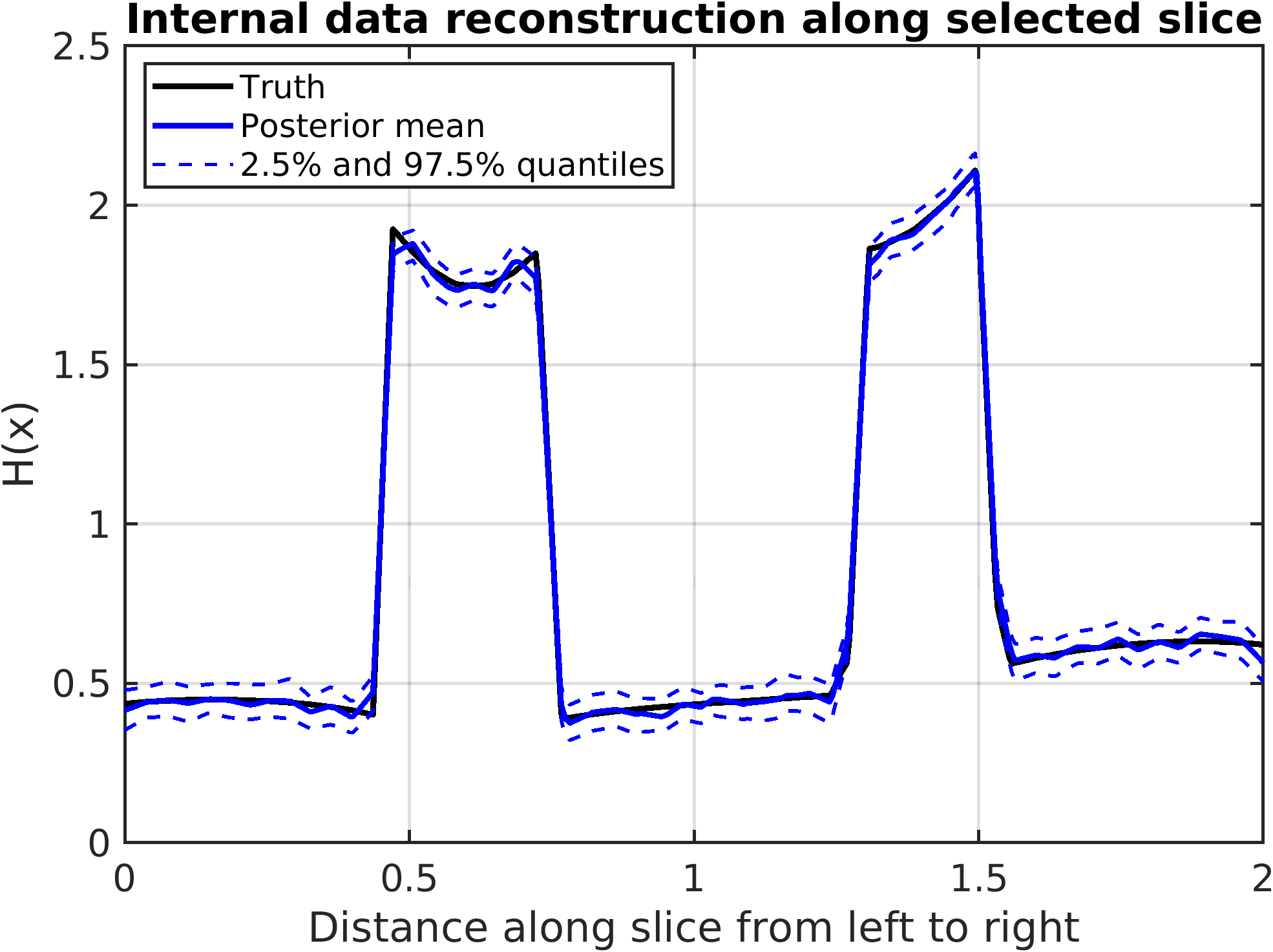}
\end{tabular}
\caption{Posterior mean reconstructions for QPAT (4\% relative noise): The first column shows the true absorption coefficient and the empirical mean of the pushedforward samples of $\gamma$. The second column shows the lower $(2.5\%)$ and upper $(97.5\%)$ credible regions from the obtained pushforward samples. The last column denotes the credible interval along a horizontal slice for $\gamma$ and along the same slice for $H$ respectively for comparision.}
\label{fig:qpat-eit-mean}
\end{figure}

\begin{figure}[htbp]
\centering

\begin{tabular}{cccc}
\includegraphics[width=0.25\textwidth]{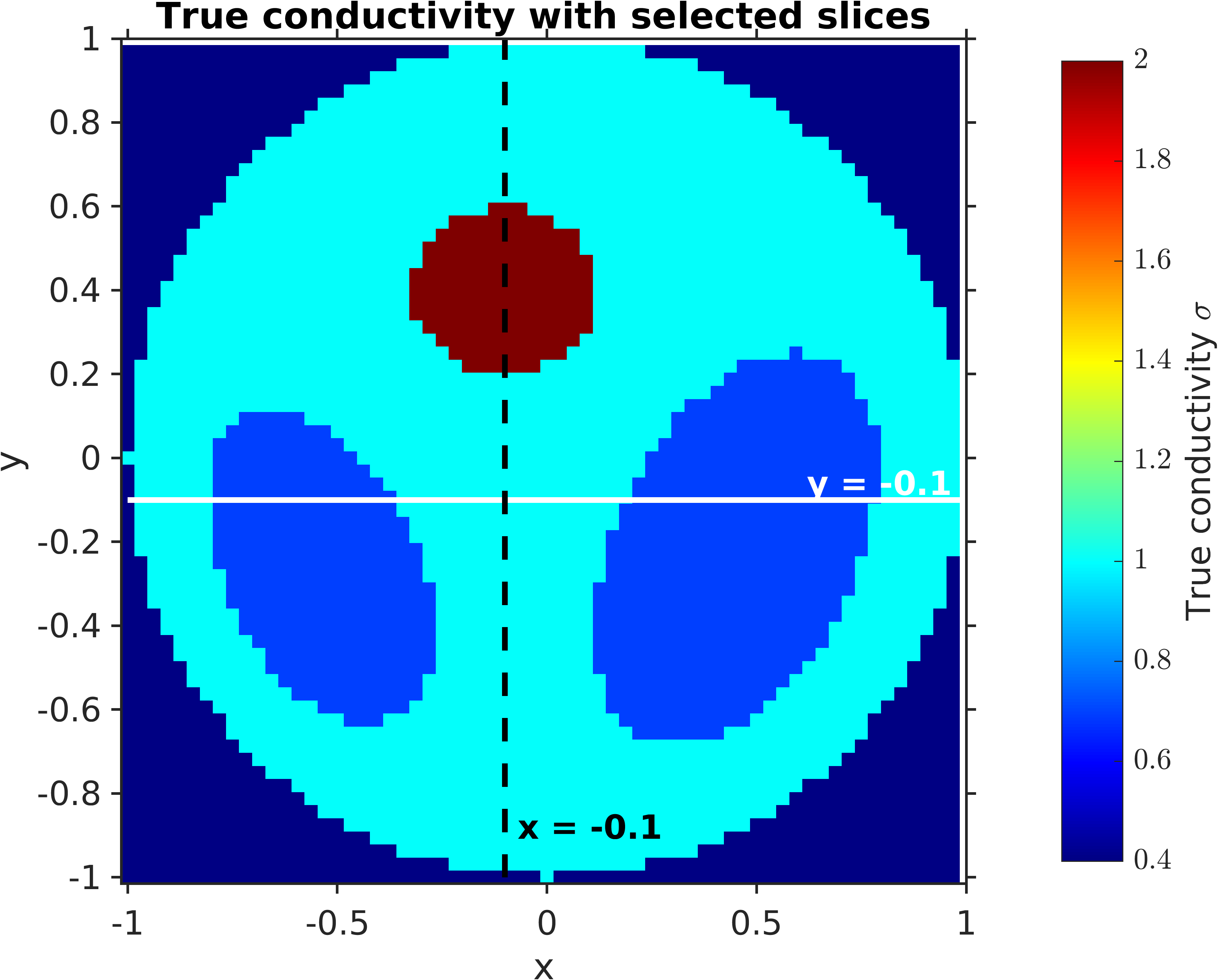}
&
\includegraphics[width=0.25\textwidth]{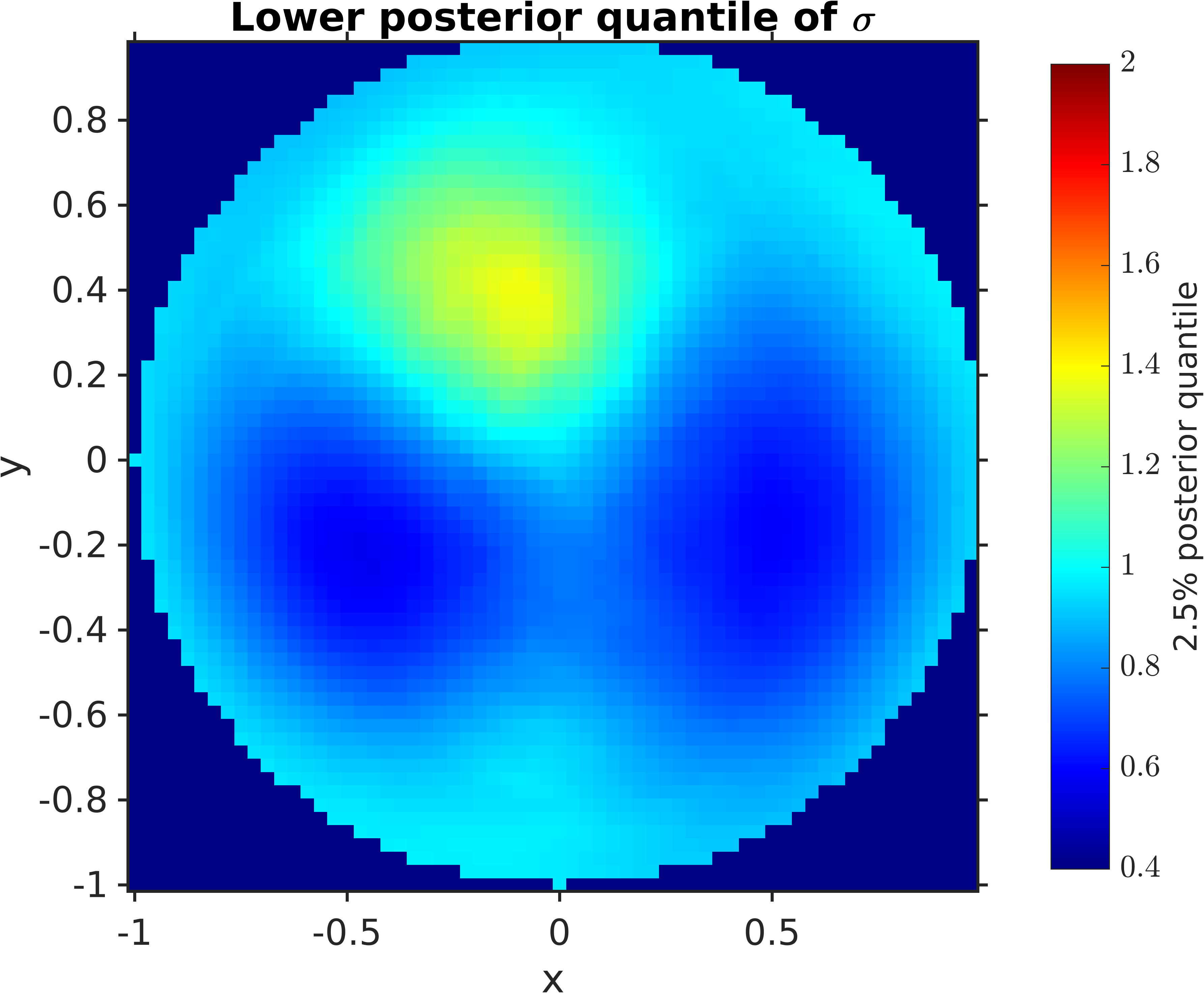}
&
\includegraphics[width=0.28\textwidth]{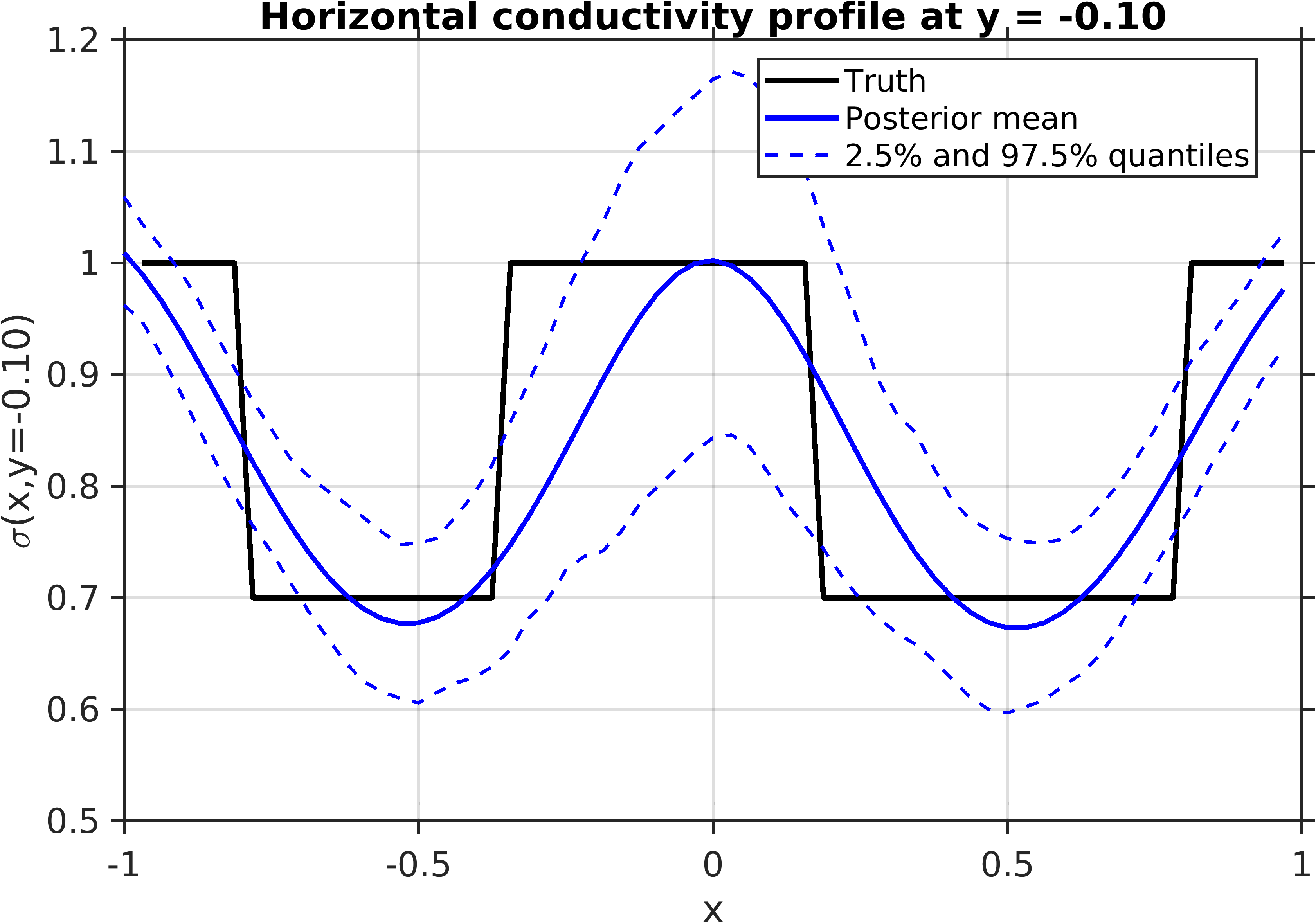}
\\
\includegraphics[width=0.25\textwidth]{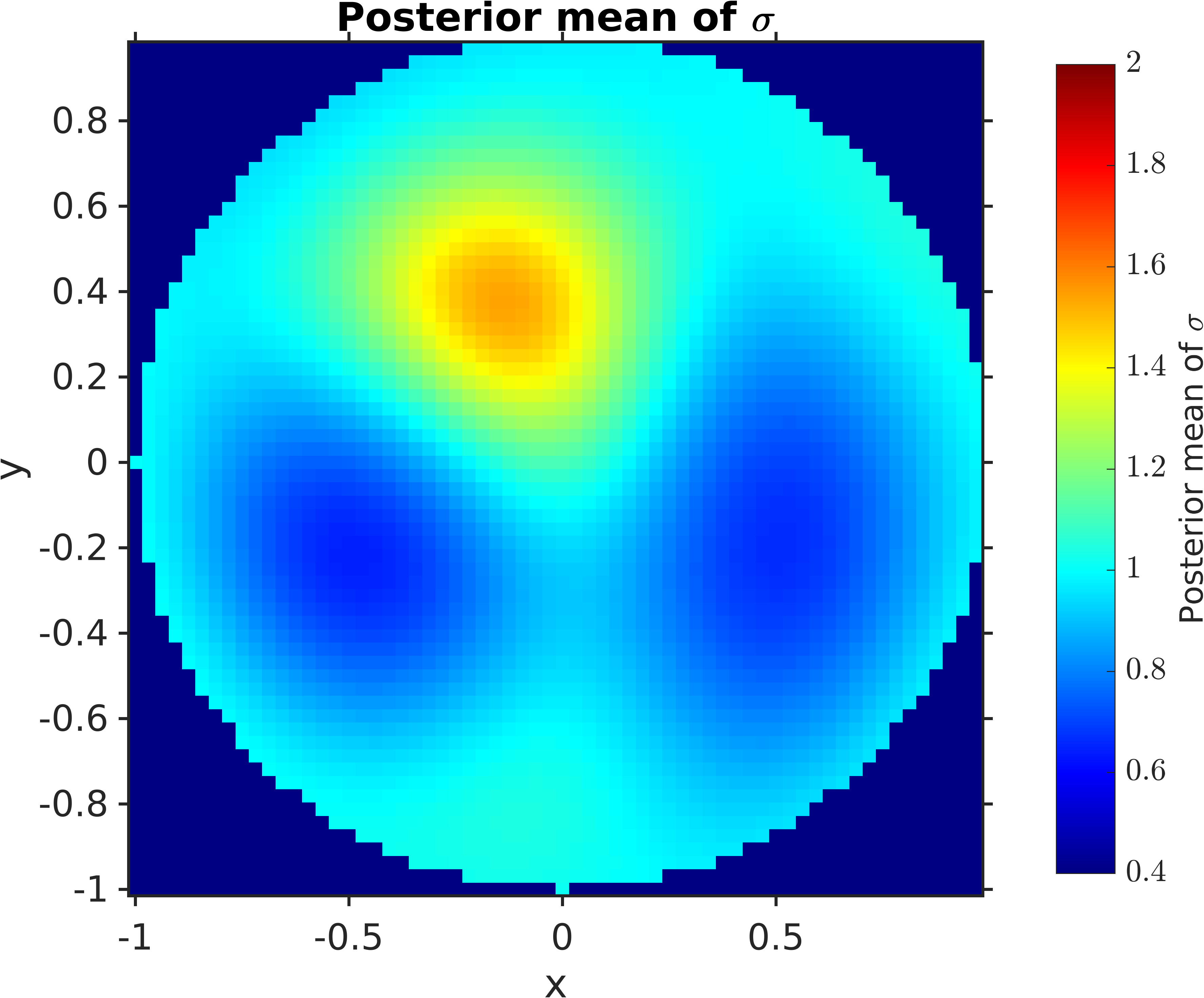}
&
\includegraphics[width=0.25\textwidth]{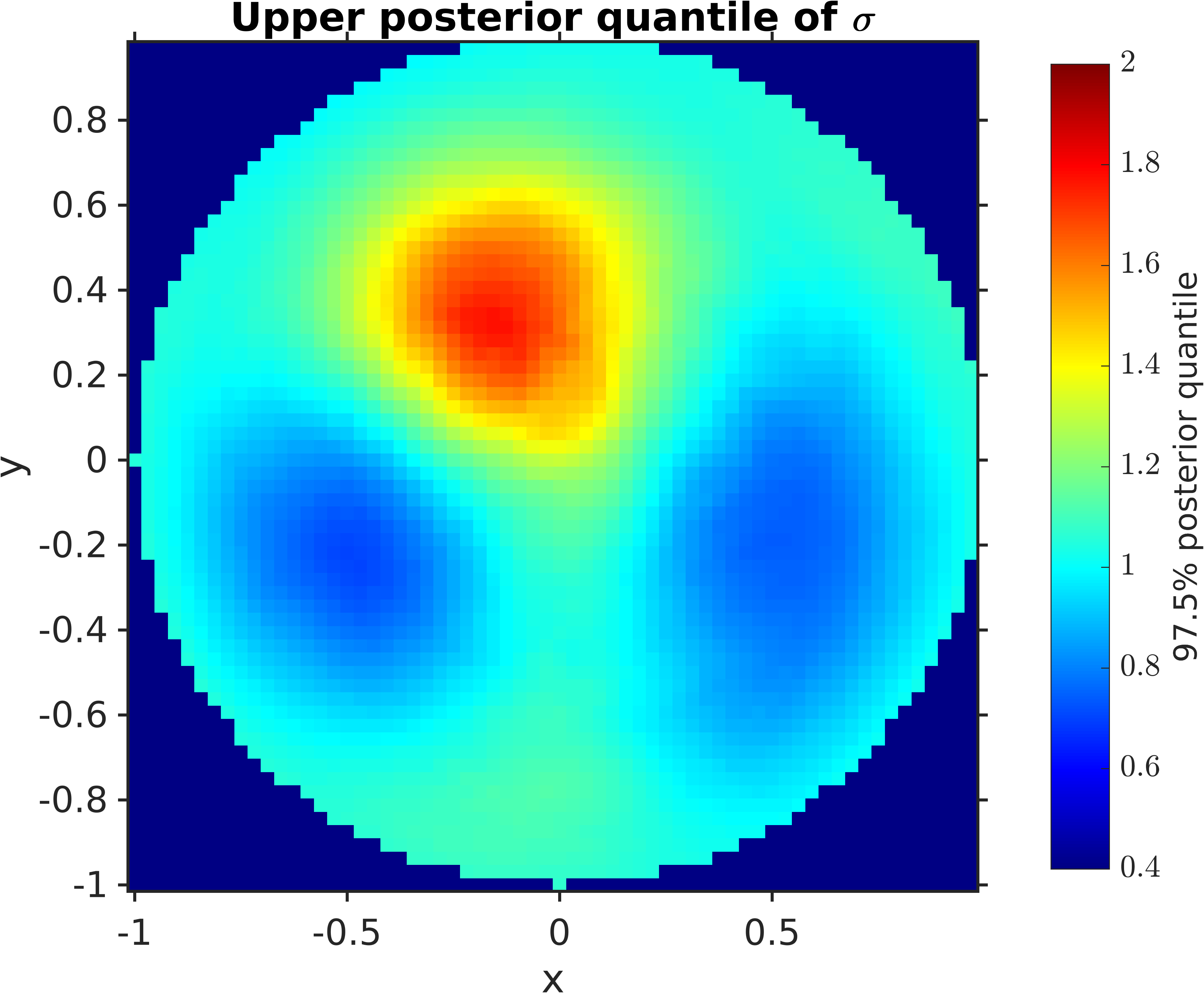}
&
\includegraphics[width=0.28\textwidth]{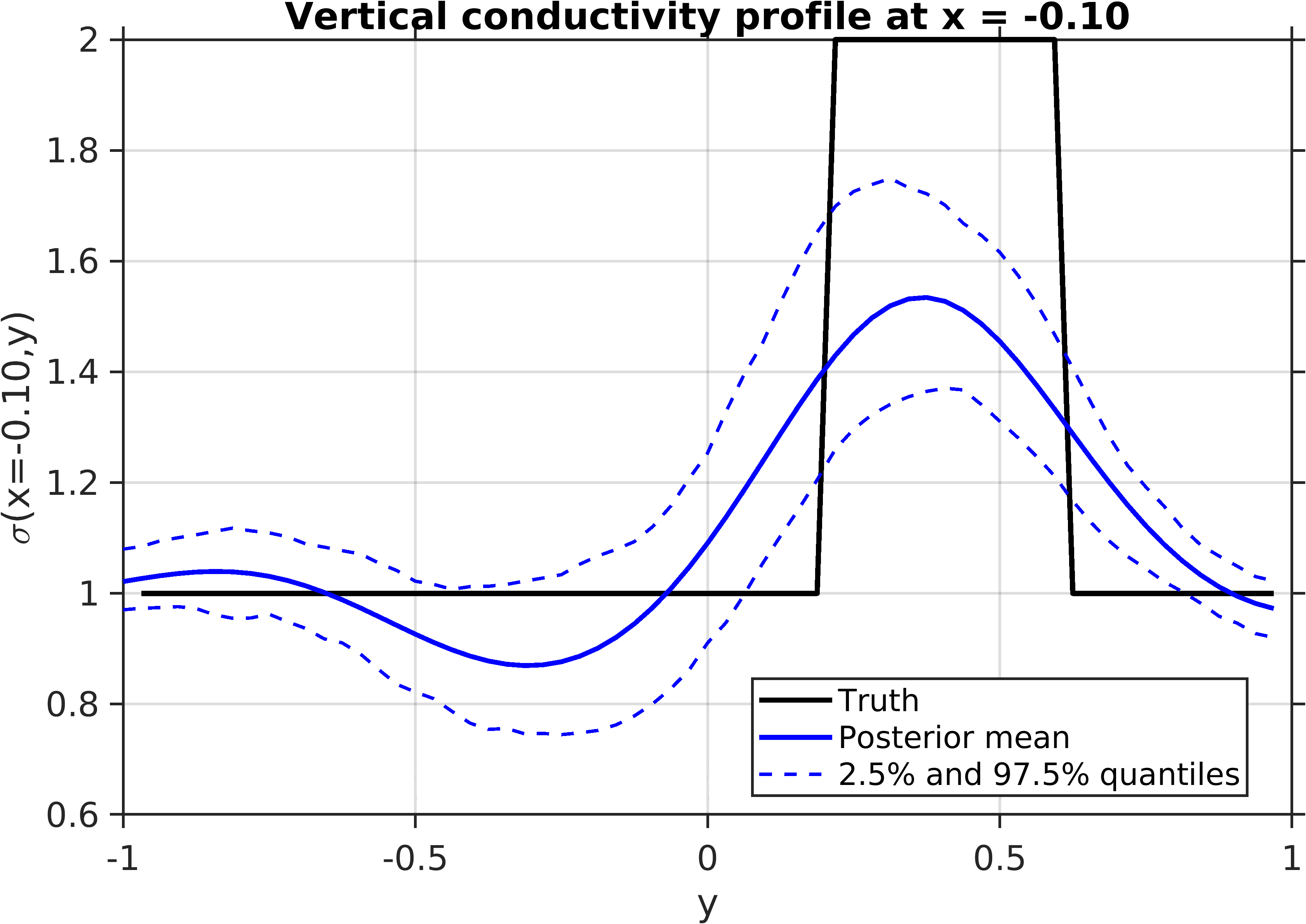}
\end{tabular}
\caption{Posterior reconstructions for EIT using D-bar method (2\% relative noise): The first column has the true conductivity and the mean of the posterior pushforward samples, the second column has the lower ($2.5\%$) and upper ($97.5\%$) credible regions plotted using 100 posterior samples. The last column shows the credible interval along the horizontal and vertical slices respectively.}
\label{fig:dbar-eit-mean}
\end{figure}

\begin{figure}[htbp]
\centering

\begin{tabular}{cccc}
\includegraphics[width=0.25\textwidth]{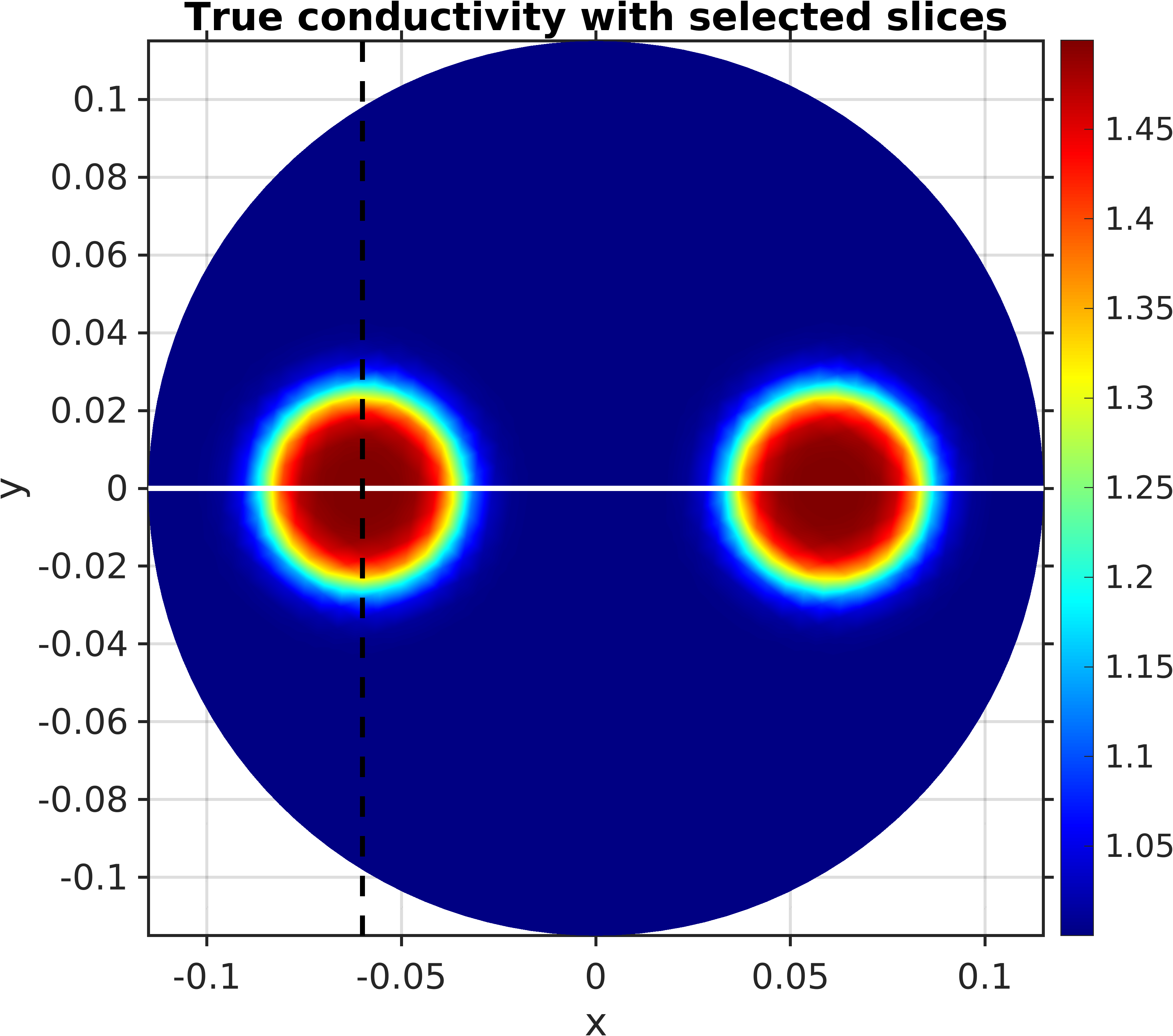}
&
\includegraphics[width=0.25\textwidth]{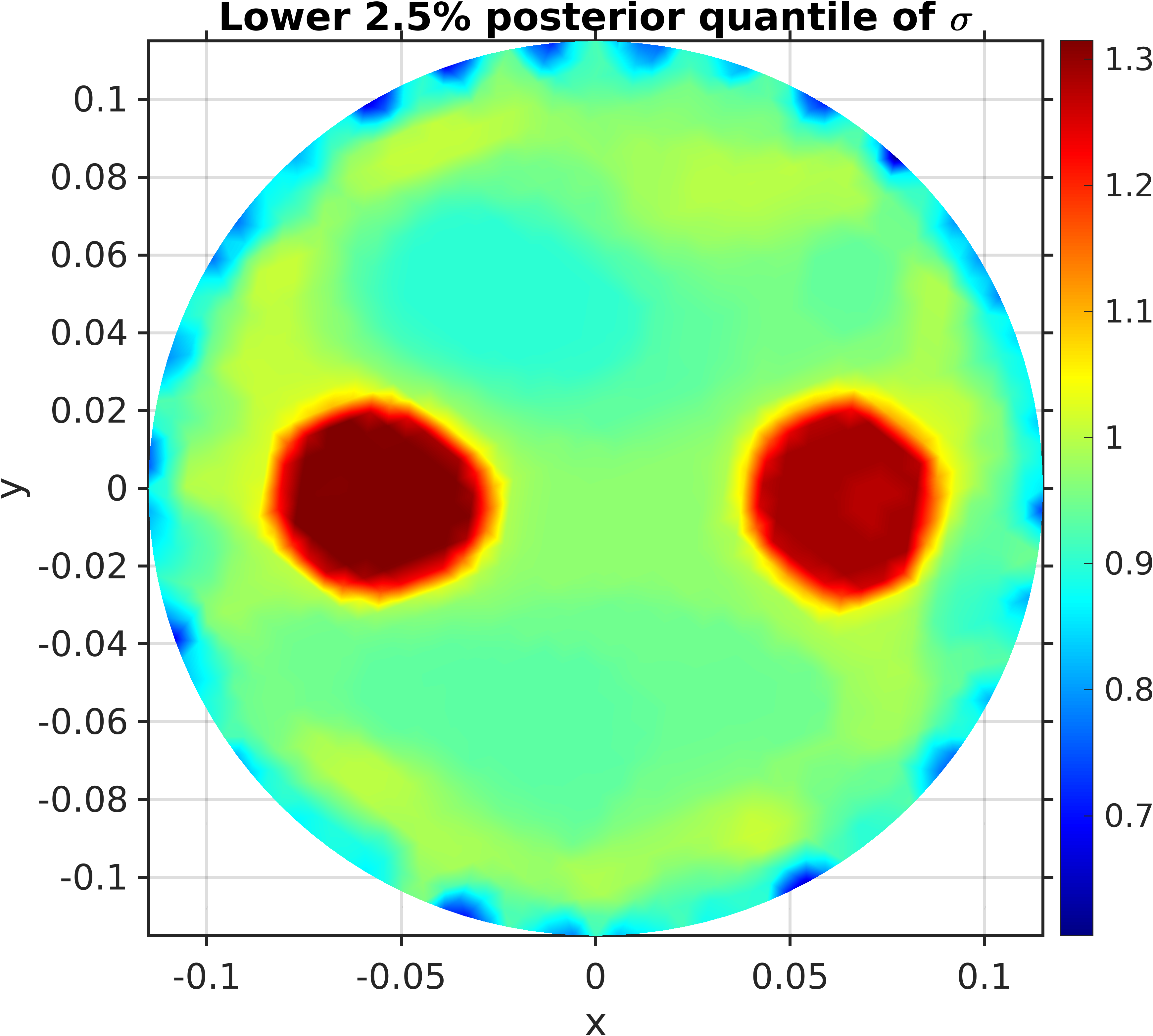}
&
\includegraphics[width=0.30\textwidth]{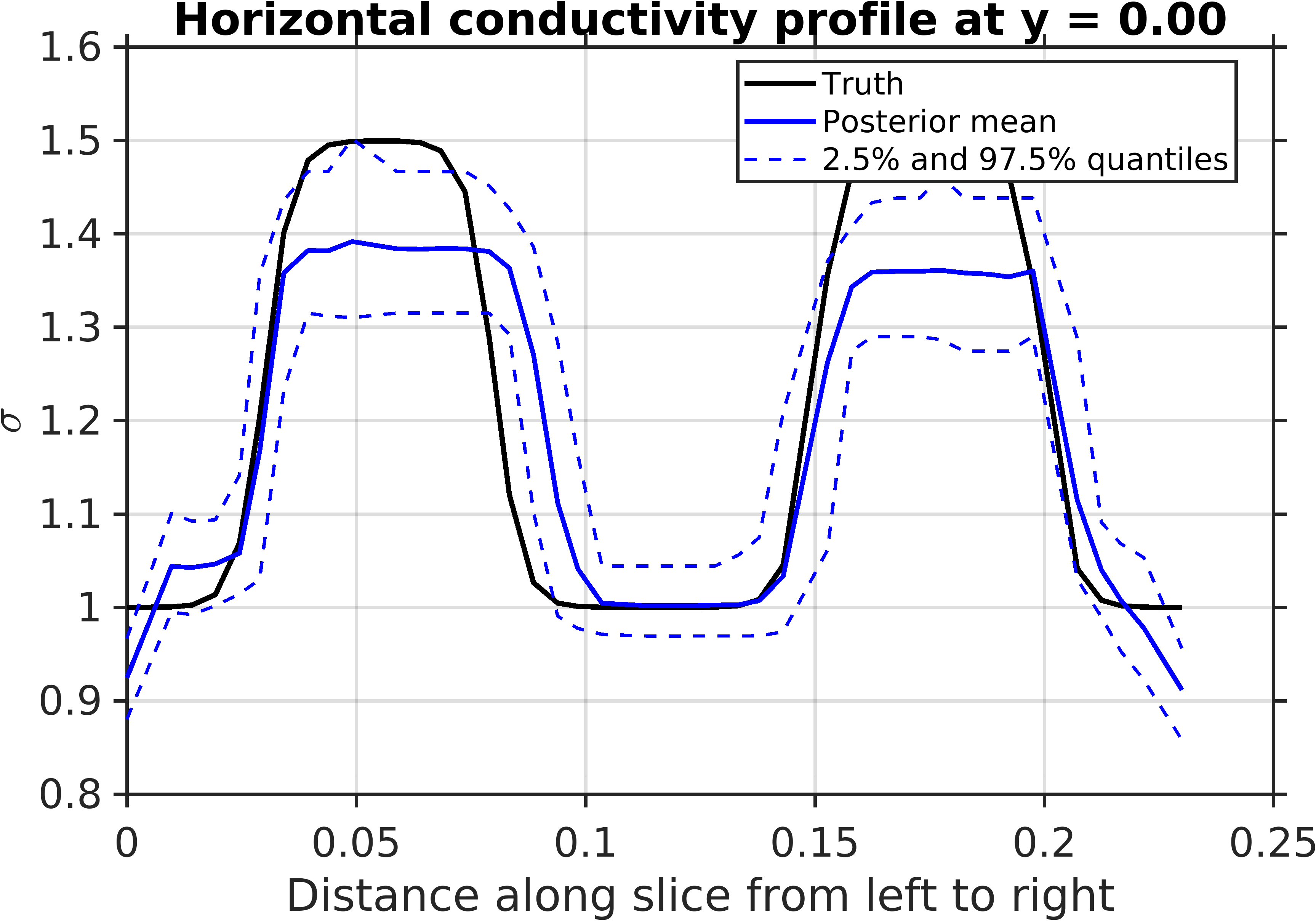}
\\
\includegraphics[width=0.25\textwidth]{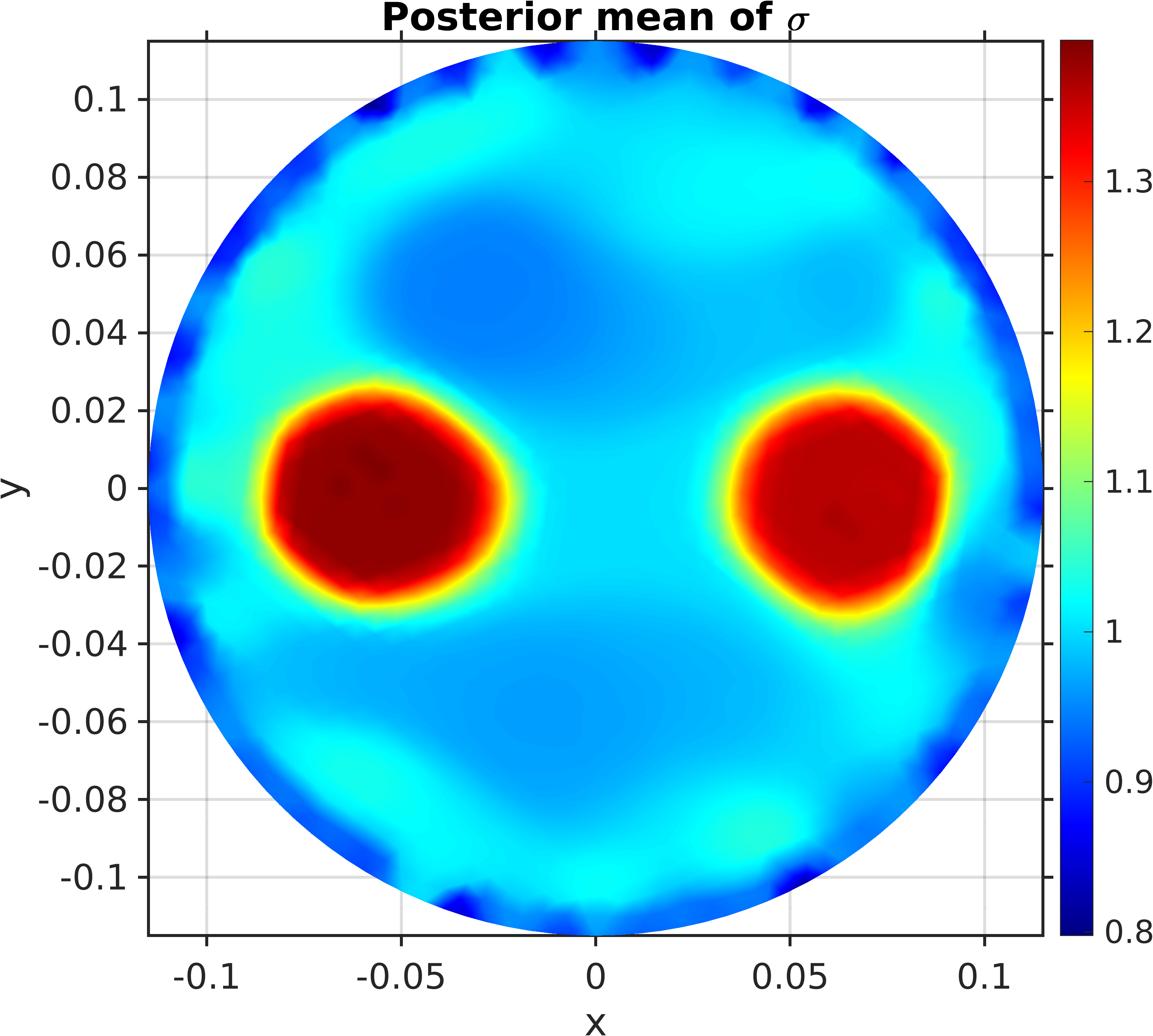}
&
\includegraphics[width=0.25\textwidth]{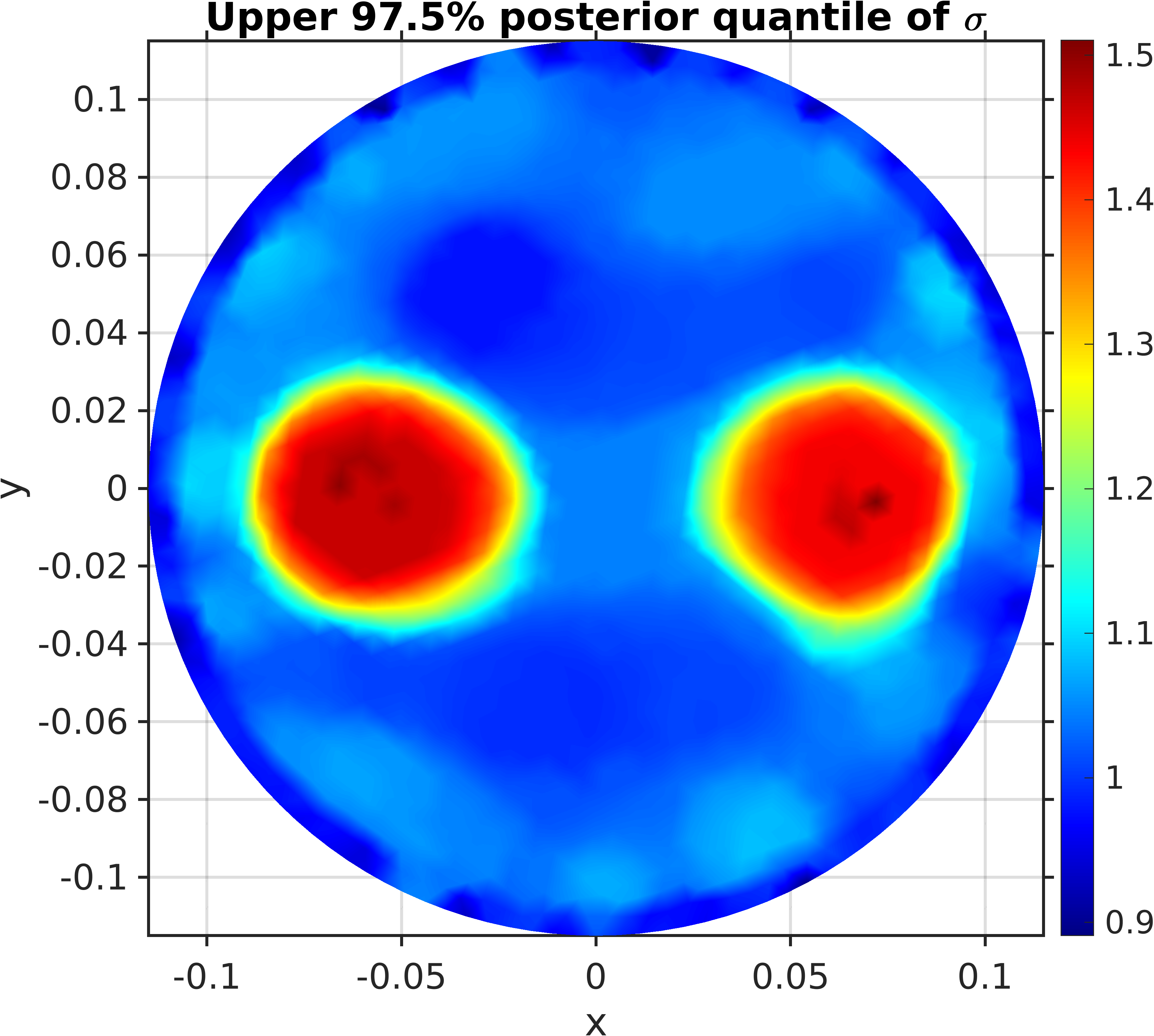}
&
\includegraphics[width=0.30\textwidth]{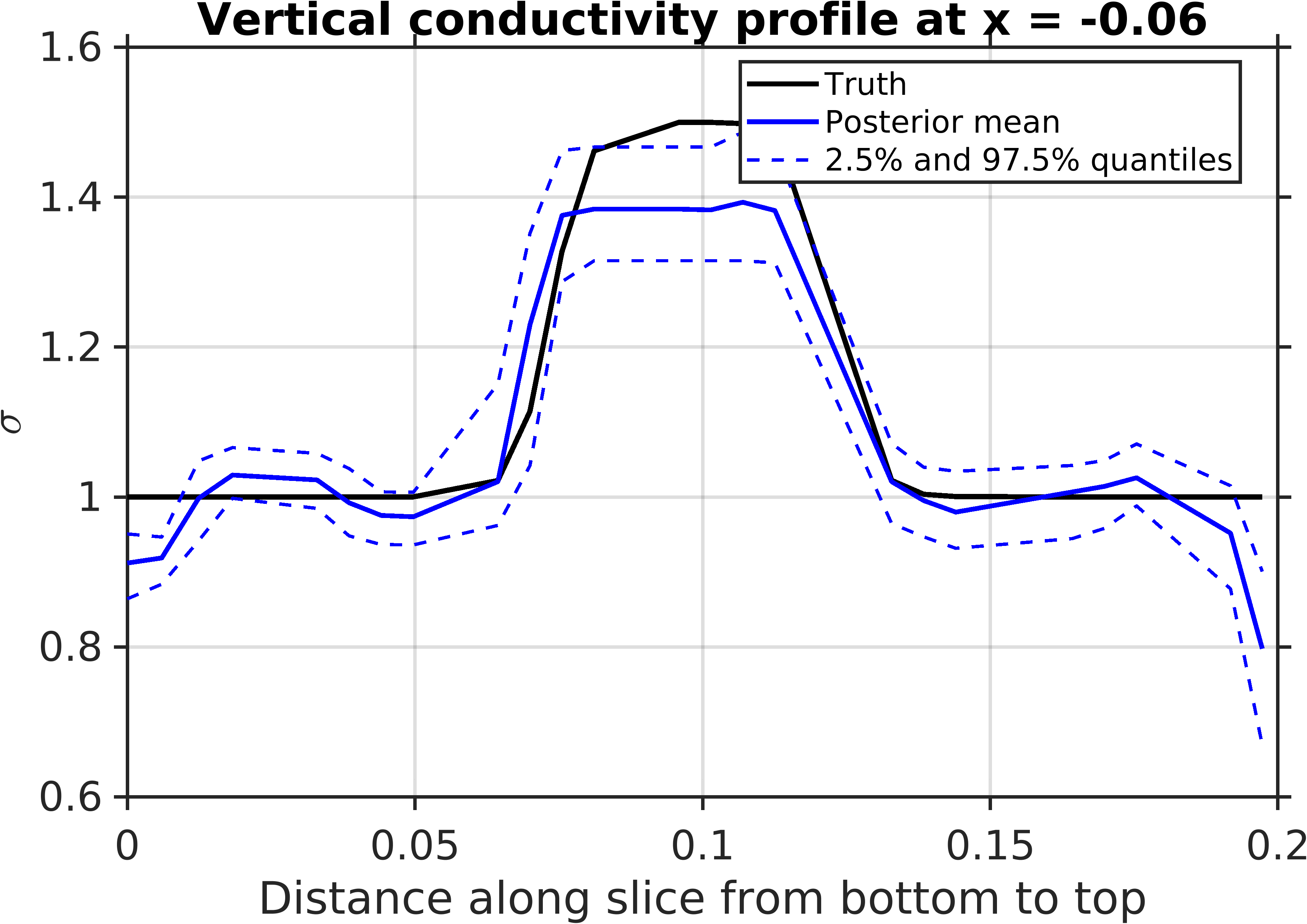}
\end{tabular}
\caption{Posterior reconstructions for EIT using OOEIT method (2\% relative noise): The first column has the true conductivity and the mean of the posterior pushforward samples, the second column has the lower ($2.5\%$) and upper ($97.5\%$) credible regions plotted using 100 posterior samples. The last column shows the credible interval along the horizontal and vertical slices respectively.}
\label{fig:ooeit-eit-mean}
\end{figure}
\subsection{Implementation for EIT}
\subsubsection{Bayesian inference in measurement space}
We begin by noting that, D-bar is based on the continuum model, whereas OOEIT uses the Complete Electrode Model (CEM) solved with the finite element method. As such, the relevant measurement operator for the D-bar method is the NtD operator (equivalent to measuring the DtN operator), whereas for OOEIT, the forward model gives the current-voltage data as measured on the boundary which models the physical data obtained in experiments. In either case, the measurements are obtained at the boundary and we will refer to such data as the boundary operator data. 
In line with our description in section \ref{sec:priorDtN} we generate a Gaussian prior for the relevant boundary operator in the following way. First, an ensemble of $2000$ samples of the log-normal conductivity field are drawn. A Gaussian random field with Matern covariance was sampled for the log-conductivity and exponentiated to ensure positivity. The Matern covriance kernel used the same smoothness scale $\nu=1.5$ in both cases, however the correlation length parameter was chosen to be $0.15$ in D-bar and $0.05$ in OOEIT. Note that length scale parameter controls spatial variations.  For each such conductivity sample the corresponding EIT forward problem is solved, and the resulting measurement data is stored. The empirical mean and covariance of the measurement ensemble are then computed and a low-rank (rank=300 in our experiments) Gaussian approximation is then computed as per the theory outlined in section \ref{sec:priorDtN}. After observing the noisy measurements, the posterior for the relevant boundary measurement operator is computed as per eq. \eqref{formula:postDtN}. In each case $100$ exact samples of the posterior are then pushedforward through either the D-bar or the OOEIT reconstruction packages described below to obtain pushforward samples of the conductivity field. Note that since the posterior measure of the boundary operator is a Gaussian, we can sample exactly from it rather than having to resort to MCMC based methods.
\par \noindent \textbf{D-bar reconstruction and phantom}

\noindent For the D-bar experiments, the true conductivity was chosen to be the standard heart-and-lungs phantom included in the D-bar MATLAB implementation. The imaging domain is the unit disk. The phantom consists of a homogeneous background together with two lower-conductivity `lung' inclusions and a higher-conductivity `heart' inclusion. 
The D-bar method is a direct nonlinear reconstruction method for EIT. It uses the boundary operator to compute complex geometrical optics solutions, forms a nonlinear scattering transform, solves the associated \(\bar{\partial}\)-equation in the spectral variable, and recovers the conductivity on a Cartesian grid. In the present pipeline, all uncertainty enters through the posterior distribution of the NtD operator; the D-bar reconstruction itself is deterministic.

\par \noindent \textbf{OOEIT reconstruction and phantoms}

\noindent For the OOEIT experiments, 
the reconstruction tests used separate deterministic inclusion phantoms as ground truth. The conductivity consists of two smooth circular inclusions centered at \((-0.06,0)\) and \((0.06,0)\), each with radius \(0.025\). The background conductivity is $1$ while conductivity inside the inclusions is $1.5$. For each posterior electrode-data sample, conductivity is reconstructed using the Gauss-Newton solver implemented in the OOEIT package.  Starting from an initial estimate for the conductivity, the forward model is repeatedly linearized about the current iterate, producing a sequence of Jacobian matrices that relate perturbations in conductivity to changes in the electrode measurements. At each iteration, a regularized linear system is solved to compute an update to the conductivity, after which the forward problem is recomputed and the process is repeated until convergence or until the prescribed number of iterations is reached. Within the OOEIT package, we used the total variation regularization together with a positivity-promoting prior to stabilize the reconstruction and preserve boundary features.
\subsubsection{Posterior uncertainty propagation}
For both D-bar and OOEIT, the  conductivity ensemble
obtained by pushforwarding the posterior samples of the boundary operator
is used to compute the posterior mean 
$ \bar{\sigma} = \frac{1}{100} \sum_{i=1}^{100}
\sigma^{(k)}$ where $\sigma^{(k)}$ corresponds to the $k-$th pushedforward sample.
Pixelwise posterior quantiles are also computed. In particular, the \(2.5\%\), and \(97.5\%\) quantiles are used to summarize uncertainty. Additionally, horizontal and vertical slices are chosen along which the posterior values are plotted to get the $95\%$ credible intervals. The results from these experiments are plotted in figures \ref{fig:dbar-eit-mean} and \ref{fig:ooeit-eit-mean} below.
\section{Conclusion}\label{sec: conclusion}
In this work, we developed a computationally efficient Bayesian framework for nonlinear PDE inverse problems and applied it to QPAT and EIT. The proposed two-stage pushforward approach avoids MCMC sampling by combining Bayesian regression for an auxiliary variable with a deterministic reconstruction map. We rigorously justified the interpretation of the induced posterior as Bayesian posterior under the induced prior and a transformed forward map for both QPAT and EIT separately. Furthermore, we derived posterior contraction rates for both problems. The numerical results demonstrate that the method provides accurate reconstructions together with reliable uncertainty estimates at a significantly lower computational cost than conventional Bayesian approaches. Finally, we note that we although we focused on QPAT and EIT, the framework is quite general and can be extended to other nonlinear inverse problems where a suitable auxiliary variable and reconstruction map are available.
\section*{Acknowledgement} AI tools (ChatGPT) have been used for improving and refactoring MATLAB codes in this project, especially to make the Bayesian framework compatible with existing packages for EIT.
\bibliographystyle{abbrv}
\bibliography{refs}
\end{document}